\documentclass[11pt,a4paper,reqno]{amsart}

\usepackage{color,calc,graphicx,amsfonts,amsthm,amscd,epsfig,psfrag,amsmath,amssymb,enumerate,dsfont}

\usepackage{amsmath,amssymb,graphics,epsfig,color,enumerate,psfrag}
\usepackage{verbatim}
\definecolor{refkey}{gray}{.75}
\definecolor{labelkey}{gray}{.5}

\newtheorem{Theorem}{Theorem}[section]

\newtheorem{Lemma}[Theorem]{Lemma}
\newtheorem{Proposition}[Theorem]{Proposition}

\newtheorem{Remark}[Theorem]{Remark}

\newtheorem{Claim}[Theorem]{Claim}

\newtheorem{Assumption}[Theorem]{Assumption}
\newtheorem{Example}[Theorem]{Example}

\definecolor{light}{gray}{.9}

\newcommand{\cA}{\ensuremath{\mathcal A}}

\newcommand{\cF}{\ensuremath{\mathcal F}}
\newcommand{\cG}{\ensuremath{\mathcal G}}

\newcommand{\E}{{\ensuremath{\mathbb E}} }
\newcommand{\bbE}{{\ensuremath{\mathbb E}} }

\newcommand{\N}{{\ensuremath{\mathbb N}} }

\renewcommand{\P}{{\ensuremath{\mathbb P}} }
\newcommand{\bbP}{{\ensuremath{\mathbb P}} }
\newcommand{\Q}{{\ensuremath{\mathbb Q}} }
\newcommand{\bbQ}{{\ensuremath{\mathbb Q}} }
\newcommand{\R}{{\ensuremath{\mathbb R}} }
\newcommand{\bbR}{{\ensuremath{\mathbb R}} }

\newcommand{\Z}{{\ensuremath{\mathbb Z}} }
\newcommand{\bbZ}{{\ensuremath{\mathbb Z}} }

\newcommand{\e}{{\rm e} }

\usepackage{mathtools}

\let\a=\alpha    \let\d=\delta  \let\ep=\varepsilon
 \let\g=\gamma       \let\l=\lambda
      \let\o=\omega      
  \let\s=\sigma \let\t=\tau   
  
\let\D=\Delta     \let\L=\Lambda 
\let\O=\Omega 

\definecolor{darkgreen}{rgb}{0,0.6,0}

\newcommand{\mic}{\textcolor{black}}
\newcommand{\ale}{\textcolor{black}}

\title[Regularity of biased 1D random walks  in random environment]{Regularity of biased 1D random walks in random environment}

\author[A.\ Faggionato]{Alessandra Faggionato}
\address{Alessandra Faggionato.
  Dipartimento di Matematica, Universit\`a di Roma ``La Sapienza''.
 P.le Aldo Moro 2, 00185 Roma, Italy}
\email{faggiona@mat.uniroma1.it}

\author[M.\ Salvi]{Michele Salvi}
\address{Michele Salvi.  {D\'epartement de Math\'ematiques Appliqu\'ees, \'Ecole Polytechnique, Route de Saclay, 91128 Palaiseau Cedex, France
and
INRA, Unit\'e MaIAGE, 78352 Jouy-en-Josas Cedex, France}}
\email{michele.salvi@polytechnique.edu}

\begin{document}
\begin{abstract} 
We study the asymptotic properties of nearest-neighbor random walks in 1d random environment under the influence of an external field of intensity $\l\in\R$. For ergodic shift-invariant environments, we show that the limiting velocity $v(\l)$ is always increasing and that it is everywhere analytic except at most in two points $\l_-$ and $\l_+$. When $\l_-$ and $\l_+$ are distinct, $v(\l)$ might fail to be continuous. We refine the assumptions in \cite{Z} for having a recentered CLT with diffusivity $\sigma^2(\l)$ and give explicit conditions for $\sigma^2(\l)$ to be analytic. For the random conductance model we show that, in contrast with the deterministic case, $\sigma^2(\l)$ is not monotone on the positive (resp.~negative) half-line and that it is not differentiable at $\l=0$. For this model we also prove the Einstein Relation, both in discrete and continuous time,  extending the result of  \cite{LD16}.

\bigskip

\noindent  \emph{AMS  subject classification (2010 MSC)}: 
60K37, 
60Fxx, 
82D30.

\smallskip


\thanks{The present work was financially supported  by  PRIN 20155PAWZB "Large Scale Random Structures" and the European Union's Horizon 2020 research and innovation programme under the Marie Sklodowska-Curie Grant agreement No 656047.
 }
\end{abstract}


\bigskip

\maketitle

\section{Introduction}
The response of a system to an external field of intensity $\l$ is relevant in many applications. In particular, one is interested in the quantitative and qualitative behavior of some large-scale quantities when $\l$ varies. As an example we mention linear response theory, where the first order $\l$--expansion of the observed quantities is analyzed (see e.g.~\cite{KTH,R09}).

The above issues have been considered both for dynamical systems and for stochastic systems. 
For stochastic systems whose evolution depends on a random environment (modeling some structural disorder), one could further ask how the disorder influences the response. 
Here we consider the special case of 1d nearest--neighbor RWRE's, and focus on the $\l$--dependence of the asymptotic velocity $v(\l)$ and  the diffusion coefficient $\s^2(\l)$. Some non--rigorous results in this direction are provided in \cite{NJB1}. According to \cite{NJB1}, differently from  the higher dimensional case treated in \cite{NJB2}, the presence of disorder in one dimension can make $v(\l)$ and $\s^2(\l)$ irregular. This picture is confirmed by some of our rigorous results.

In this paper we investigate  the  behavior of the quantities $v(\l)$ and $\sigma^2(\l)$ as functions of the parameter $\l$. In particular, we focus on their monotonicity, differentiability and analyticity, and we derive the Einstein Relation for the random conductance model (RCM) extending the result of \cite{LD16}.  
The advantage of working with nearest neighbor walks on the one-dimensional lattice is that $v(\l)$ and $\s^2(\l)$ have an explicit representation in terms of suitable series (see \cite{Z}). 

\smallskip

Before entering in the details of our results we describe some previous contributions on related problems. 
The  monotone behavior of the speed of RWRE's in  dimension $d\geq 2$ has been considered in several papers. One of the most interesting and most studied models is that of a walk on the infinite supercritical percolation cluster, where the speed has been proved to be positive up to a critical value of $\l$ and equal to zero above this threshold \cite{BGP03, FH13}. This non-monotone nature of the speed as a function of the bias has been also recently observed for walks among elliptic conductances \cite{BGN}.
Results concerning   the continuity of the speed  have been obtained  e.g.~for a random walk in a one--dimensional percolation model \cite{GMM} and for  the 1d Mott random walk \cite{FGS1}  (in \cite{GMM} also the differentiability has been studied).
The behavior of one dimensional RWRE's that are transient but with zero-speed has been studied in \cite{ESTZ, ESZ, KKS} for i.i.d.~jump probabilities and in \cite{BS17} for the RCM.
The continuity of the diffusion matrix  at $\l=0$ has been derived in \cite{MP2} for diffusions in random environment. In the context of random walks on groups, analyticity of the speed and of the asymptotic variance has been proved in  \cite{HMM}, while  in \cite{Gouz}  the same result is proved  with dynamical ideas for general hyperbolic groups.
Finally, a particular attention has been devoted  to  the linear response of  RWRE's  for a weak  bias (also in higher dimension). 
This has lead to the proof of the Einstein relation, which claims the equivalence between the derivative of $v(\l)$ at $\l=0$ and the diffusion coefficient of the unperturbed process $\sigma^2(0)$ (see \cite{FGS2,GGN,GMP,G,KO1,KO2,LD16,LR,L1,L2}).
 
\smallskip 
 
 We now describe our results and outline the paper. We analyze in detail biased 1d nearest--neighbor RWRE's both in discrete and in continuous time. The discrete time model is introduced in Section \ref{discreto}, while the continuous time case is introduced in Section \ref{continuo}. In Section \ref{assunzioni} we state our main assumptions and introduce the concept of a reflection invariant environment. The random conductance model (RCM), which is of particular interest in what follows, appears in Section \ref{conduco}.

Section \ref{latte} is dedicated to the study of the asymptotic velocity $v_X(\l)$ of a generic  discrete time random walk as a function of the external bias, while Section \ref{miele} treats the  velocity $v_Y(\l)$ in the continuous time case.
We show that $v_X(\l)$ is analytic everywhere with exception of at most two values $\l_-\leq \l_+$, it is strictly increasing on $(-\infty, \l_-) $ and $(\l_+,+\infty)$ and it is zero on $(\l_-, \l_+)$ (cf.~Proposition \ref{marte1}). The same holds for $v_Y(\l)$ (cf.~Proposition \ref{marte10}). 
If $\l_-=\l_+$ then $v_X(\l)$ is continuous on all $\bbR$, see Proposition \ref{tazzina}. The corresponding result for $v_Y(\l)$ appears in Proposition \ref{tazzinabis}. Sections \ref{riflessioned} and \ref{riflessionec} deal with the reflection invariant environment case. Finally we exhibit examples with an irregular behavior of the speed. Example \ref{lorenzo1} (resp.~Example \ref{lorenzo3}) shows a pathological model for which $v_X(\l)$ (resp.~$v_Y(\l)$) is not continuous in $\lambda_{\pm}$. Even when the environment is given by a (genuinely random) i.i.d.~sequence of jump probabilities, $v_X(\l)$ is not differentiable at the two points $\l_-<\l_+$, see Example \ref{lorenzo15}. The discrete time RCM with i.i.d.~genuinely random  conductances   has speed  $v_X(\l)$ without second derivative at $\l_+=\l_-=0$, see Example \ref{lorenzo2}. In continuous time, the RCM regularizes (Example \ref{antonio}), but we provide another elementary model (see Example \ref{lorenzo4}) where $v_Y(\l)$ does not have second derivative at $\l_+=\l_-=0$.


Section \ref{sectioneinstein} is dedicated to the proof of the Einstein relation for the biased RCM. 
In Theorem \ref{einstein} 
we provide a shorter alternative proof to the one appearing in \cite{LD16} and also extend the result to more general hypothesis.

In Section \ref{caniefigli} we move to the study of the central limit theorem (CLT) when the random walk is ballistic, restricting to the case of discrete time. Theorem \ref{baldo} extends the CLT discussed in \cite{Z} and Proposition \ref{sviluppo} provides an alternative description of the diffusion coefficient $\s^2(\l)$. In Proposition \ref{analiticita} we give some conditions that guarantee analyticity of $\s^2(\l)$. 
Finally, in Subsection \ref{trilli} we gather some sufficient conditions for the CLT to hold that are easier to verify. Applications are given in Example \ref{pierpaolo2007} for the case of an environment given by i.i.d.~jump probabilities and  in Section  \ref{diff_RCM} for the RCM with i.i.d.~conductances.

In Section \ref{diff_RCM} we focus on $\sigma^2(\l)$ for the RCM with i.i.d.~conductances.
In Theorem \ref{suppongo} we explicitly calculate $\s^2(\l)$ and in Proposition \ref{unpofru} we prove that $\s^2(\l)$ is continuous everywhere, it is analytic on $\bbR\setminus\{0\}$ but it is not differentiable at $\l=0$ if the conductances are genuinely random.  Moreover, we show that, differently from the case of deterministic conductances, $\s^2(\l)$ is neither monotone on $[0,\infty)$ nor on $(-\infty,0]$.

Sections \ref{lorenzo1_esteso} and \ref{lorenzo4_esteso} are the discussions of Examples \ref{lorenzo1} and \ref{lorenzo4}, respectively. In the Appendix we collect some technical proofs.

%
%
%
%
%
%

\section{Models}\label{modelli}
In this section we introduce our nearest--neighbor random walks on $\bbZ$ and fix our notation. We distinguish between discrete time random walks and continuous time random walks. 

\subsection{Discrete time random walks}\label{discreto}
We first consider discrete time random walks on $\bbZ$ in random environment. To this aim we let  $\O=(0,1)^\bbZ$ be the space of environments endowed with the product topology and with a probability measure $\bbP$ ($\bbE$ will denote the associated expectation).
We write   $\o = (\o_x^+)_{x \in \bbZ}$ for a generic element of $\O$ and set $\o_x^- := 1-\o_x^+$. We introduce then an external force, or bias, of intensity $\l\in\R$. This results in modifying the environment in the following way: fixed   $ \l \in \bbR$ we define
\begin{equation}\label{nizza}
\o^-_x(\l):= \frac{ \o^-_x \e^{-\l}}{  \o^-_x \e^{-\l}+ \o^+_x \e^\l   } \,, \qquad \o^+_x(\l):= \frac{ \o^+_x \e^{\l}}{  \o^-_x \e^{-\l}+ \o^+_x \e^\l }\,.
\end{equation}

Given a realization  $\o$  of the environment, $(X_n ^{\o, \l})_{n \geq 0}$   will be  the discrete time random walk starting at the origin  and jumping from $x$ to $x\pm 1$ with probability $\o_x^\pm(\l)$. 
 We write $ P^{\,\o, \l}_0$ and $ E^{\,\o, \l}_0$  for  the associated probability and expectation, respectively, with the convention that we will write simply $X_n$ when dealing with  
$ P^{\,\o, \l}_0$,  $ E^{\,\o,\l}_0$. 
In particular, we have 
\[  P^{\,\o, \l}_0\bigl( X_{n+1}=x-1 \,|\, X_{n}=x\bigr) = \o_x^{-}(\l) \,, \qquad P^{\,\o, \l} _0\bigl( X_{n+1}=x+1 \,|\, X_{n}=x\bigr) =\o_x^+(\l)\,.\]
Finally we define 
\begin{equation}\label{datta} \rho_x(\l):=\frac{\o^-_x(\l)}{\o^+_x(\l)}\,.
\end{equation}

When $\l=0$ we will refer to the unperturbed random walk and 
omit the index $\l$, writing simply $X_n ^{\o}$, $ P^{\o}_0$ and $E^{\o}_0$,  $\rho_x$. Note in particular that we have 
\begin{equation}
\rho_x(\l)=\frac{\o^-_x(\l)}{\o^+_x(\l)}=\frac{\o^-_x}{\o^+_x}\e^{-2\l}= \rho_x \e^{-2\l} \,.
\end{equation}
We think  of $X_n ^{\o,\l}$ as a perturbation  of $X_n ^{\o}$ due to  the presence of an external  field of intensity $\l$.
\subsection{Continuous  time random walks}\label{continuo} 
When considering continuous time random walks, we let 
$ \Theta= \bigl(\, (0,+\infty)\times (0,+\infty)\, \bigr)^\bbZ$  be the space of environments endowed with the product topology and with a probability measure $\bbP$ ($\bbE$ will denote the associated expectation). We let $\xi:= \bigl(\, (r_x^-, r_x^+)\, \bigr)_{x \in \bbZ}$ be a generic element of $\Theta$. 
Fixed  $ \l \in \bbR$ we   set
\[
r_x^{-}(\l)= r_x^- \e^{- \l}\,, \qquad r_x^{+}(\l)= r_x^+ \e^{ \l}\,.
\]
Then 
\ale{$(Y_t ^{\xi, \l})_{t \geq 0}$} will denote the continuous time random walk on $\bbZ$ starting at the origin, having nearest--neighbour jumps    with probability rate for  a jump from $x$ to $x\pm 1$ given by $r_x^\pm(\l)$. Below (cf. Assumption \ref{uffina})  we will give conditions assuring that $(Y_t ^{\xi, \l})_{t \geq 0}$ is well defined a.s.~(i.e.~no explosion takes place a.s.).

We denote by $P^{\,\xi, \l}_0$ and $ E^{\,\xi, \l} _0$  the associated probability and expectation \mic{ and also in this case we will just write $Y_t$ for the random walk when it appears inside $P^{\,\xi, \l}_0$ or $ E^{\,\xi, \l} _0$}. In particular, we have  
\[ 
P_0^{\,\xi, \l} \bigl( \, Y_{t+dt} =x-1  \,|\, Y_t=x \, \bigr) = r_x^-(\l)dt\,, \qquad 
P_0^{\,\xi,\l} \bigl( \, Y_{t+dt} =x+1  \,|\, Y_t=x \, \bigr) = r_x^+(\l) dt\,.\]
When $\l=0$ we will refer to the unperturbed random walk and 
omit the index $\l$, writing simply \ale{$Y_t ^{\xi}$,} $ P^{\,\xi}_0$ and $ E^{\,\xi}_0$. 

\smallskip

We note that the associated  discrete time version recording only the jumps (the so called \emph{jump process}), has probability for a jump from $x $ to $x\pm 1$ given by 
\begin{equation}\label{corro1} \o_x^\pm (\l):=
\frac{ r_x^\pm (\l)}{ r_x^- (\l)+r_x^+ (\l)}
= \frac{
\frac{r_x^\pm }{r_x^-+r_x^+} \e^\l
}{\frac{r_x^- }{r_x^-+r_x^+} \e^{-\l}+\frac{r_x^+ }{r_x^-+r_x^+} \e^{\l}   }=
\frac{ \o^\pm_x \e^{\pm\l}}{  \o^-_x \e^{-\l}+ \o^+_x \e^\l   }
\,,
\end{equation}
where
\begin{equation}\label{corro2} \o_x^\pm := \frac{r_x^\pm }{r_x^-+r_x^+} \,.\end{equation}
Note  that identities \eqref{nizza} are satisfied. 
In particular,  the jump process associated to the perturbed continuous time random walk $(Y_t ^{\xi, \l})_{t \geq 0}$ is the 
perturbed Markov chain associated to the jump process of $(Y_t ^{\xi })_{t \geq 0}$.
When dealing with  continuous time random walks we will keep the definitions \eqref{corro1}, \eqref{corro2} and define $\rho _x$ and $\rho_x(\l)$ according to \eqref{datta}. Note that 
\[ \rho_x= \frac{r_x^-}{r_x^+}\,,\qquad  \rho_x(\l)= \rho_x \e^{-2\l} 
\,.
\]

\subsection{Assumptions on the environment}\label{assunzioni}
For both the discrete time and the continuous time random walks we will \textit{always} make the following assumption:
\begin{Assumption}[Main Assumption]\label{uffina}
The law $\bbP$ of the environment is   stationary and ergodic with respect to shifts   and    $\E[\log \rho_0]$ is well defined, with $\pm\infty$ as possible values.
 \end{Assumption}
 We point out that Assumption \ref{uffina} corresponds to  Assumption 2.1.1 in \cite{Z}. 
\begin{Lemma}
Under Assumption \ref{uffina},  for the continuous time random walk a.s.~explosion does not take place and therefore  $Y^{\xi, \l}_t$  is well defined for all times $t$. \end{Lemma} 
  \begin{proof}
 Let $\o$ be defined by \eqref{corro2} and consider the associated biased discrete time random walk $(X^{\o, \l}_n)_{n\geq 0}$. Then one can introduce the continuous time random walk as a random time change of  $(X^{\o, \l}_n)_{n\geq 0}$ by imposing that, once arrived at site $x$, the random walk remains at $x$ for an exponential random time with mean $\bigl( r_x^+(\l)+ r_x^-(\l)\bigr)^{-1}$. Take now $M>0$ such that $\bbP(r_0^+(\l)+ r_0^-(\l) \leq M)>0 $. Consider the random set $A:=\{ x\in \bbZ\,:\, r_x^+(\l)+ r_x^-(\l)\leq M\}$. Then, by the  ergodic theorem and Assumption \ref{uffina},  a.s. $A\cap \bbZ_+$ and $A\cap \bbZ_-$ are infinite sets.   By Assumption \ref{uffina} and \cite[Thm.~2.1.2]{Z}, a.s. the random walk $(X^{\o, \l}_n)_{n\geq 0}$ visits an half-line of $\bbZ$. Hence $\lim _{n\to\infty} f(n)=+\infty$ where   $f(n) :=|\{ k \in \{0,1,\dots,n\}\,:\, X_k^{\o, \l}\in A\}|$. 
 As a consequence, for almost all realizations of $(X^{\o, \l}_n)_{n\geq 0}$, 
 when we condition to the realization of $(X^{\o, \l}_n)_{n\geq 0}$ we get that  the $n$--th jump of the continuous time random walk takes place at a random time $t_n$  which stochastically dominates the sum of $f(n)$ i.i.d. exponential random variables with mean $M^{-1}$. Therefore, $t_n$ goes to infinity as $n\to \infty$ a.s. 
   \end{proof}
  
While Assumption \ref{uffina} will always hold in what follows, in order to build special counterexamples  we will sometimes make the following assumption (this will be clearly specified in the text):

\begin{Assumption}[Reflection invariance - discrete time case]\label{jesus}
The law of the environment is left invariant by the spatial reflection with respect to the origin, i.e.~ by the transformation   $(\o_x^+)_{x\in \bbZ} \mapsto (\o_{-x}^-)_{x\in \bbZ}$.
\end{Assumption} 
 
Also in the continuous time setting we will sometimes consider models with a special symmetry:
\begin{Assumption}[Reflection invariance - continuous time case]\label{jesus_cont} 
The law of the environment is left invariant by the spatial reflection with respect to the origin, i.e.~ by the transformation   $\bigl(\, (r_x^-, r_x^+)\, \bigr)_{x \in \bbZ}   \mapsto  \bigl(\, (r_{-x}^+, r_{-x}^-)\, \bigr)_{x \in \bbZ}  $.
\end{Assumption}

\subsection{Random conductance model}\label{conduco} 
\ale{In what follows, when  
referring to  \emph{random conductances}, we will mean 
 a family of positive random variables $(c_x)_{x\in \bbZ}$, stationary and ergodic w.r.t.~shifts.  }
  The number $c_x $ is  called  the \emph{conductance}  of the edge $\{x,x+1\}$. 
 Then the discrete time \emph{random conductance model} (RCM) is given by  the random walk $X_n^{\o}$ where $\o^+_x := \frac{c_x}{c_{x-1}+c_x}$ and $\o^-_x :=
\frac{c_{x-1}}{c_{x-1}+c_x}$ for all $x \in \bbZ$. The random walk $X_n^{\o,\l}$ represents the biased discrete time  RCM.  
The continuous  time RCM  is given by  the random walk  \ale{$Y_t^{\xi }$}  where $\xi= \bigl(\, (r_x^-, r_x^+) \, \bigr)_{x \in \bbZ}\in \Theta $  satisfies    $r_x^+=c_x =  r_{x+1}^-$ for all  $x \in \bbZ$.    The random walk  \ale{$Y_t^{\xi ,\l}$}  represents the biased continuous time RCM.
\ale{
\begin{Remark}\label{fieno} 
Our  main assumption (see Assumption \ref{uffina}) for the random conductance model is satisfied if  $\E[|\log c_0|]<\infty$ since $\log \rho_0 = \log c_{-1} - \log c_0$. We point out that  $\E[|\log c_0|]<\infty$ if $c_0, \; c_0^{-1}  \in L^1(\bbP)$.
\end{Remark}}



\section{Discrete time asymptotic velocity $v_X(\l)$}\label{latte}

As in \cite[Eq.~(2.1.7)-(2.1.8)]{Z} we set
\begin{align*}
\bar S(\l)
	:=\frac{1}{\o^+_0(\l)}+\sum_{i=1}^\infty \frac{1}{\o^+_{-i}(\l)}\prod_{j=0}^{i-1}\rho_{-j}(\l)	\,,\qquad 
\bar F(\l)
	:=\frac{1}{\o^-_0(\l)}+\sum_{i=1}^\infty \frac{1}{\o^-_{i}(\l)}\prod_{j=0}^{i-1}\rho^{-1}_{j}(\l).
\end{align*}

\begin{Proposition}\label{ferb}  \cite[Theorem 2.1.9]{Z}
  The limit $v_X(\l):= \lim_{n\to \infty}\frac{X_n^{\o,\l} }{n} $ exists $\bbP\otimes P_0^\o$--a.s., is not random  and is characterized as follows:
\begin{itemize}
	\item[(a)]
$\E[\bar S(\l)]<\infty
	\Longrightarrow v_X(\l)=\frac{1}{\E\left[\bar S(\l)\right]}$;
	\item[(b)]
$\E[\bar F(\l)]<\infty
	\Longrightarrow v_X(\l)=-\frac{1}{\E\left[\bar F(\l)\right]}$;
	\item[(c)]
$\E[\bar S(\l)]=\infty \mbox{ and }\E[\bar F(\l)]=\infty
	\Longrightarrow v_X(\l)=0\,.$
\end{itemize}
\end{Proposition}

\begin{Lemma}\label{buddha}
It holds
\begin{align*}
\bar S(\l)
	=1+2\sum_{i=0}^\infty \rho_0\rho_{-1}\cdots\rho_{-i}\,\e^{-2\l (i+1)}\,,\qquad
\bar F(\l)
	=1+2\sum_{i=0}^\infty \rho_0^{-1}\rho_{1}^{-1}\cdots\rho_{i}^{-1}\,\e^{2\l (i+1)}\,.
\end{align*}
\end{Lemma}
\begin{proof}
We observe that $\frac{1}{\o^+_x(\l)}=1+\rho_x\e^{-2\l}$.	Hence
\begin{align*}
\bar S(\l)
	&=1+\rho_0\e^{-2\l} + \sum_{i=1}^\infty (1+\rho_{-i}\e^{-2\l})\rho_0\rho_{-1}\cdots\rho_{-i+1}\e^{-2i\l}\\
	&=1+\rho_0\e^{-2\l} + \sum_{i=1}^\infty \rho_0\rho_{-1}\cdots\rho_{-i+1}\e^{-2i\l}+\sum_{i=1}^\infty \rho_0\rho_{-1}\cdots\rho_{-i+1}\rho_{-i}\e^{-2(i+1)\l}\\
	&=1 + \sum_{i=1}^\infty \rho_0\rho_{-1}\cdots\rho_{-i+1}\e^{-2i\l}+\sum_{i=0}^\infty \rho_0\rho_{-1}\cdots\rho_{-i+1}\rho_{-i}\e^{-2(i+1)\l}\, .
\end{align*}
The proof for $\bar F(\l)$ is similar.
\end{proof}

To describe some regularity properties of the asymptotic velocity $v_X(\l)$ we introduce the thresholds $\l_-$ and $\l_+$ as follows:
\[  \l_-:= \sup \{ \l\in \bbR  \,:\, v_X(\l) <0\}\,, \qquad \l_+:= \inf \{ \l \in \bbR \,:\, v_X(\l) >0\}\,,\]
with the convention that  \ale{$\sup \emptyset =-\infty$ and $\inf \emptyset =+\infty$. }

\begin{Proposition}\label{marte1}  
 The velocity 
$v_X(\l) $ is   increasing in $\l$ and $\l_-\leq \l_+$. Moreover,    $v_X(\l)$ is strictly increasing and analytic on $(-\infty , \l_-)$ and on $(\l_+, +\infty)$, while it is zero on $(\l_-, \l_+)$.
 \end{Proposition}
 \begin{Remark}\label{marte2} Due to the above proposition,  $v_X(\l)$ is analytic everywhere with possible exception at $\l_-, \l_+$, where it can be irregular (even discontinuous, see Section \ref{sec_piatti}). 
 \end{Remark}
\begin{proof}[Proof of Proposition \ref{marte1}] Due to the representation given in Lemma \ref{buddha} one gets that 
the function $\bar S : \bbR \to [0, +\infty]$  is decreasing,  and that the function $\bar F: \bbR \to [0,+\infty]$ is increasing.  Combining this observation with  Proposition \ref{ferb}, one gets that $v_X(\l) $ is   increasing in $\l$, thus implying that $\l_-\leq \l_+$ and that $v_X(\l)=0$ on $(\l_-,\l_+)$.  Note that $\l_+ =  \inf \{ \l \in \bbR \,:\, \bbE[ \bar S(\l) ] <\infty \}$. Hence, for $ \l > \l_+$ we have $\bar S(\l) <\infty$ $\bbP$--a.s.  This property and the form of $\bar S$ given in Lemma \ref{buddha}  allow to conclude that, given  $ \l_+ < \l<\l'$, 
 $ \bar S(\l) > \bar S(\l')$ $\bbP$--a.s. Since $\bbE [  \bar S(\l)]$ and  $\bbE [  \bar S(\l')]$ 
are finite, we then conclude that  $\bbE [  \bar S(\l)]> \bbE [  \bar S(\l')]$ and therefore that $v_X(\l) < v_X( \l')$  (cf.  Item (a) in Proposition \ref{ferb}). In a similar way, one proves that $v_X(\l)$ is strictly increasing in $(-\infty, \lambda_-)$.

It remains to prove that $v_X(\l)$ is analytic  on $(-\infty , \l_-)$ and on $(\l_+, +\infty)$. We show its analyticity on  $(\l_+, +\infty)$, the case $(-\infty , \l_-)$ is similar. Since $ v_X(\l)=\frac{1}{\E[\overline{S}(\l)]} $  and $\E[\bar S(\l)]$ is finite and positive on 
$(\l_+, +\infty)$, it is enough to prove that the map $\l \mapsto h(\l):=\bbE[\bar S(\l)]$ is analytic on $(\l_+, +\infty)$. 
This  follows  from  Lemma \ref{latticini} in Appendix, which is based on 
the Theorem of  Pringsheim-Boas (cf. \cite[Thm.~3.1.1]{KP}).
\end{proof}

 As already pointed out, there are models for which  $v_X(\l)$ is discontinuous at $\l_-$ or $\l_+$. On the other hand, if $\l_-=\l_+$ this cannot happen:
\begin{Proposition} \label{tazzina}
If  $\l_-=\l_+$, then    $v_X(\l)$ is continuous. Moreover, it must be $\l_-=\l_+=\E[\log\rho_0]$ and $v_X(\l_\pm)=0$.
\end{Proposition}
\begin{proof}
The continuity for $\l>\l_+$ and for $\l<\l_-$ is given by Proposition \ref{marte1}. Let us  check that $\l_-=\l_+=\E[\log\rho_0] $ and that $v_X(\l)$ is continuous at 
 $\l=\E[\log\rho_0]$ when this value is finite.
First of all we notice that $\E[\bar S(\l)]$ is finite iff $\sum_{i=0}^\infty \E[\rho_0\cdots\rho_i]{\rm e}^{-\l(i+1)}$ is finite \ale{(see Lemma \ref{buddha}). Since, by Jensen's inequality,}
\begin{align}\label{sanpaolo}
\E[\rho_0\cdots\rho_i]
	=\E[{\rm e}^{\log\rho_0+\dots+\log\rho_i}]
	\geq {\rm e}^{\E[\log\rho_0+\dots+\log\rho_i]}
	= {\rm e}^{\E[\log\rho_0](i+1)}\,,
\end{align}
we always have that $\l_+\geq \bar \l:=\E[\log\rho_0]$.

Analogously, $\E[\bar F(\l)]$ is finite iff $\sum_{i=0}^\infty \E[\rho_0^{-1}\cdots\rho_i^{-1}]{\rm e}^{\l(i+1)}$ is finite \ale{ and }
\begin{align*}
\E[\rho_0^{-1}\cdots\rho_i^{-1}]	
	=\E[{\rm e}^{-\log\rho_0-\dots-\log\rho_i}]
	\geq {\rm e}^{\E[-\log\rho_0-\dots-\log\rho_i]}
	= {\rm e}^{-\E[\log\rho_0](i+1)}\,,
\end{align*}
so that it is always true that $\l_-\leq \bar \l:=\E[\log\rho_0]$. It follows that, if $\l_-=\l_+$, it must be $\l_-=\l_+=\bar\l=\E[\log\rho_0]$. If $\E[\log\rho_0]=\pm \infty$, there is nothing left to prove.

\smallskip

From now on we assume that $\bar \l=\E[\log\rho_0]$ is finite.  From  the  observations we just made, it also follows that $v_X(\bar \l)=0$, since at $\bar\l$ both $\E[\bar S(\l)]$ and $\E[\bar F(\l)]$ are infinite.
W.l.o.g.~we assume now by contradiction that there is a discontinuity to the right of $\bar\l$. If this is the case, we must have $\lim_{\varepsilon\to 0}v_X(\bar\l+\varepsilon)>0$, or equivalently $\lim_{\varepsilon\to 0}\E[\bar S(\bar \l+\varepsilon)]<\infty$. But this is in contradiction with the following:
\begin{align*}
\lim_{\varepsilon\to 0} \E[\bar S(\bar \l+\varepsilon)]
	&= 1+2\lim_{\varepsilon\to 0}\sum_{i=0}^\infty \E[\rho_0\cdots\rho_i]{\rm e}^{-(\bar \l+\varepsilon)(i+1)}\\
	&\geq 1+2\lim_{\varepsilon\to 0}\sum_{i=0}^\infty {\rm e}^{\bar \l(i+1)}{\rm e}^{-(\bar \l+\varepsilon)(i+1)}\\
	&=1+2\lim_{\varepsilon\to 0}\sum_{i=0}^\infty {\rm e}^{-\varepsilon(i+1)}
	=1+2\lim_{\varepsilon\to 0}\frac{{\rm e}^{-\varepsilon}}{1-{\rm e}^{-\varepsilon}}
	=\infty\,,
\end{align*}
where for the inequality we have used \eqref{sanpaolo}.
\end{proof}

\subsection{Reflection invariance case}\label{riflessioned}
In the particular case of reflection invariant environments (cf.~Assumption \ref{jesus}) we have the following:


\begin{Proposition}\label{fiumicino}
Suppose Assumption \ref{jesus} to be satisfied. Then it holds
\begin{align}\label{flauto_bis} 
v_X(\l)=-v_X(-\l)\,.
\end{align}
In particular, $v_X(0)=0$ and, if $v_X$ has $n$-th derivative at $0$ with $n$ even, then  this derivative must be $0$.
Moreover, 
the following dichotomy holds for $\l\geq 0$:
\begin{align*}
\begin{cases}
\E[\bar S(\l)]<+\infty&\Longrightarrow \quad v_X(\l)=\frac{1}{\E[\overline{S}(\l)]}>0\,,\\
\E[\bar S(\l)]=+\infty &\Longrightarrow \quad v_X(\l)=0\,.
\end{cases}
\end{align*}
\end{Proposition}
\begin{proof} Identity \eqref{flauto_bis} follows by symmetry, while the identity $v^{(n)}_X(0)=0$ (for even $n$'s) follows from \eqref{flauto_bis}.
 By Proposition \ref{ferb} to get the dichotomy  it is enough to check that $\E[\bar F]=\infty$ for $\l\geq 0$. Since $\rho_0=
 \frac{\o_0^-}{\o_0^+}$ and $\rho_0^{-1}=\frac{\o_0^+}{\o_0^-}$, by Assumption  \ref{jesus} we have that $\rho_0$ and $\rho_0^{-1}$ have the same law. In particular, $\E[\log \rho_0]=0 $ and therefore $\E[\log(\rho_0^{-1}\rho_{1}^{-1}\cdots\rho_{i}^{-1})]=0$. By Jensen's inequality and Lemma \ref{buddha} \ale{we conclude that }
\begin{align*}
\E[\bar F]
	&=1+2\sum_{i=0}^\infty \E[\rho_0^{-1}\rho_{1}^{-1}\cdots\rho_{i}^{-1}]\,\e^{2\l (i+1)}\\
	&\geq1+2\sum_{i=0}^\infty \e^{\E[\log(\rho_0^{-1}\rho_{1}^{-1}\cdots\rho_{i}^{-1})]}\,\e^{2\l (i+1)}=1+2\sum_{i=0}^\infty \e^{2\l (i+1)} =\infty.\qedhere
\end{align*}
\end{proof}

\subsection{Examples of models with irregular asymptotic velocity $v_X(\l)$}\label{sec_piatti}
We conclude this section with three examples.  In Example \ref{lorenzo1}    $v_X(\l)$ is not continuous at $\l_+$. This example is rather exotic and if one is interested in models violating  e.g. the   analiticity of $v_X(\l)$, then it is enough to consider random walks with i.i.d.~and genuinely random  $\omega_i^+$ (see Example \ref{lorenzo15})  or the  RCM  with i.i.d.~and genuinely random conductances (see Example \ref{lorenzo2}). In Example \ref{lorenzo15} $v_X(\l)$ is not differentiable in $\l_- $, $\l_+$ (in this case $\l_-<\l_+$), while in 
  Example \ref{lorenzo2}  $v_X(\l)$ has not second derivative at $0=\l_-=\l_+$.

\begin{Example}\label{lorenzo1}
$v_X(\l)$ is in general not continuous as in  the following model. Fixed the \ale{parameters}  $A>0$ and $\g>2$,  we  first introduce   the random variables $r(k,k-1)$ and $r(k,k+1)$, $k\in \bbZ$.   We set  $r(k,k-1):=A$ for all $k\in\Z$. To define  $r(k,k+1)$ we proceed as follows. We let $\tilde \tau=(\tilde\tau_k)_{k\in\Z}$ be a renewal point process on $\bbZ$  \ale{ such that  $\tilde\tau_0=0$ and, for $k\not=0$, 
 $\P(\tilde{\tau}_{k+1}- \tilde{\tau}_{k} \geq j)=c/j^\gamma$ for any $j  \in\N_+$.  Here,   $c$ is   the appropriate renormalizing constant and $\N_+$ is the set of positive integers.}
 We write $\t = (\tau_k)_{k\in\Z}$ for the renewal point process  given by the $\bbZ$--stationary version of $\tilde \tau$ (see Section \ref{lorenzo1_esteso}). 
For $k\in\Z$,  we set 
\[
r(k,k+1):= \begin{cases}
1\quad\mbox{if }- k\notin \tau\,,\\
2 \quad\mbox{if }-k\in \tau\,.
\end{cases} 
\]
Finally, we take $\o_k^+:=\frac{r(k,k+1)}{r(k,k-1)+ r(k,k+1)}$, $\o_k^-:=\frac{r(k,k-1)}{r(k,k-1)+ r(k,k+1)}$.
Then Assumption \ref{uffina} is satisfied,  $v_X(\l)> 0$ for $\l\geq \l_+=\frac 12 \log A $, while $v_X(\l)\leq 0$  for $\l<\l_+$. In particular,   $v_X(\l)$ has a discontinuity at $\l_+$. \ale{In addition,  $\l_-$ is finite and $\l_-<\l_+$.} 
\end{Example}
The discussion of the above example is given in Section \ref{lorenzo1_esteso}.

\begin{Example}\label{lorenzo15} Consider the case of i.i.d.~$\o_i^+$'s such that $\bbE[\log \rho_0]$ is well defined (this
assures that Assumption \ref{uffina} is satisfied). Then $\l_+=\tfrac 12\log \E[\rho_0] $ and $\l_-=-\tfrac 12\log \E[1/\rho_0]$. Moreover, if the $\o_i^+$'s are  genuinely random  then $\l_-<\l_+$ and $v_X(\l)$ is not differentiable at $\l_+$ if $\l_+$ is finite.  The same holds  for $\l_-$.
\end{Example}
\begin{proof}[Discussion of Example \ref{lorenzo15}]
By applying Proposition \ref{ferb} and  Lemma \ref{buddha}, we have  $\l_+=\tfrac 12\log \E[\rho_0] $, $\l_-=-\tfrac 12\log \E[1/\rho_0]$ and 
\begin{align}\label{joe}
v_X(\l)=
\begin{cases}
\frac{1-\E[\rho_0]\e^{-2\l}}{1+\E[\rho_0]\e^{-2\l}}& \mbox{if } \l>\l_+\,,\\
0 &  \mbox{if } \l_-\leq \l \leq \l_+\,,\\
-\frac{1-\E[1/\rho_0]\e^{2\l}}{1+\E[1/\rho_0]\e^{2\l}} &\mbox{if } \l<\l_-\,.
\end{cases}
\end{align}
In particular, $v_X(\l)$ is continuous.
Let us now restrict to genuinely random variables $\o_i^+$.
By Jensen inequality we  have $\l_-<\l_+$.
We also  notice that the right derivative of $v_X(\l)$ for $\l\geq \l_+$ is $\frac{4\E[\rho_0]\e^{-2\l}}{(1+\E[\rho_0]\e^{-2\l})^2}$, which is equal to $1$ in $\l_+$  if $\l_+$ is finite. Hence $v_X(\l)$ is not differentiable in $\l_+$. Similar considerations hold for $\l_-$.
\end{proof}

\medskip

Let us now consider a discrete time  RCM with \ale{with 
conductances $(c_x)_{x\in \bbZ}$ (see Section \ref{conduco})}.  We collect some observations which will be used in the next Example \ref{lorenzo2}.
Since $\rho_x=\frac{c_{x-1}}{c_x}$ we have $\rho_0\rho_{-1}\cdots\rho_{-i}=\frac{c_{-i-1}}{c_0}$. By Lemma \ref{buddha}
we then get 
$\bar S(\l)=1+2\sum_{i=0}^\infty \frac{c_{-i-1}}{c_0}\e^{-2\l(i+1)}$ and 
by using the translation invariance of the conductances  we conclude that 
\begin{align}\label{scacco}
\E[\bar S(\l)]=1+2\sum_{i=1}^{\infty}\E\Big[\frac{c_0}{c_i}\Big]\e^{-2\l i}\,.
\end{align}
Finally, we note   that for the RCM Assumption \ref{jesus} is equivalent to saying that the sequences $(c_x)_{x\in\Z}$ and $(c_{-x})_{x\in\Z}$ have the same law. In particular, if the conductances are i.i.d. as in Example \ref{lorenzo2} below, Assumption \ref{jesus} is satisfied. 


\begin{Example}\label{lorenzo2}
Consider the discrete time RCM with  i.i.d.~conductances such that  $\E[c_0]<\infty$, $\E[1/c_0]<\infty$ and $c_0$ is not almost surely constant. Then  the model is reflection invariant and at $\l_{+}=\l_-=0$  $v_X(\l)$ is continuous, has first derivative but has no second derivative. 
\end{Example}
\begin{proof}[Discussion of Example \ref{lorenzo2}]
 \ale{As discussed above and due to Remark \ref{fieno},    Assumptions \ref{uffina} and  \ref{jesus} are satisfied}.
Take $\l\geq 0$ and let $A:=\E[c_0]$ and $B:=\E[1/c_0]$. Note that by Jensen's inequality $AB>1$, since $c_0$ is non deterministic. By \eqref{scacco} we have 
\begin{align*}
\E[\bar S(\l)]=1+2AB\,\frac{\e^{-2\l}}{1-\e^{-2\l}}= \frac{  1-\e^{-2\l}+2AB\e^{-2\l}}{ 1-\e^{-2\l}}\, .
\end{align*}
By  Proposition \ref{fiumicino} we conclude that $\l_+=0$ and
\begin{equation}\label{v_RCM}
v_X(\l)=\frac{1-\e^{-2\l}}{1-\e^{-2\l}+2AB\e^{-2\l}}\,, \qquad \l\geq0\,.
\end{equation}
A similar analysis can be done for $\l<0$.
By a Taylor expansion, \ale{for $\l \geq 0$} we have 
\begin{equation*}
\begin{split}
v_X(\l)
	&= \frac{ 2\l - 2 \l^2 + o(\l^2) }{ 2\l + 2AB- 4AB \l  +o(\l)    }
	=\frac{1}{AB}\cdot  \frac{  \l -  \l^2 + o(\l^2) }{  1+ (1/AB - 2 ) \l  +o(\l)    } \\
	&=\frac{1}{AB} \bigl(  \l -  \l^2 + o(\l^2)  \bigr )\Bigl(  1- \big(\frac{1}{AB} - 2 \big) \l  +o(\l)   \Bigr)
	=\frac{1}{AB} \l + \frac{AB-1}{(AB)^2} \l^2 + o(\l^2)\,. 
  \end{split}
  \end{equation*}
Since $v(\l)=-v(-\l)$, the above Taylor expansion shows that $v(\l)$ is continuous and has a first derivative at $\l=0$. On the other hand, $ \partial_{\l\l}^+v_X(0)/2 =  (AB-1)/(AB)^2 \ale{>0}$ (recall that $AB>1$). By Proposition \ref{fiumicino} we conclude that  $v_X(\cdot)$ has no second derivative at $0$. 
\end{proof}

\begin{Remark} 
 In Example \ref{lorenzo15} with  $\rho_i$ equal to a fixed constant $C$, we have $v_X(\l)=\frac{1-C\e^{-2\l}}{1+C\e^{-2\l}}$ for each $\l\in\R$. In particular, $v_X(\l)$ is everywhere analytic. Analogously, in Example \ref{lorenzo2} with  conductances   equal to  a fixed constant, then $v_X(\l)=\frac{1-\e^{-2\l}}{1+\e^{-2\l}}$ and therefore  $v_X(\l)$  is again everywhere analytic. In particular, in both cases  the emergence of the irregularity of the asymptotic velocity $v_X$ corresponds to the randomness  of the  environment.
\end{Remark}

\section{Continuous time asymptotic velocity $v_Y(\l)$}\label{miele}

\bigskip

We set 
\begin{align*}
& \hat S(\l)
	:=\frac{1}{r^+_0(\l)
	}+\sum_{i=1}^\infty \frac{1}{ r^+_{-i}(\l)}
	\prod_{j=0}^{i-1}\rho_{-j}(\l)	\,,\\
&\hat F(\l)
	:=\frac{1}{r^-_0(\l)}+\sum_{i=1}^\infty \frac{1}{r_i^-(\l)
	 }\prod_{j=0}^{i-1}\rho^{-1}_{j}(\l).
\end{align*}

\begin{Proposition}\label{ariccia}   The limit $v_Y(\l):= \lim_{t\to \infty}\frac{Y_t^{\xi,\l} }{t} $ exists \ale{$\bbP\otimes P_0^{\xi, \l}$--a.s.,} is not random  and is characterized as follows:
\begin{itemize}
	\item[(a)]
$\E[\hat S(\l)]<\infty
	\Longrightarrow v_Y(\l)=\frac{1}{\E[\hat S(\l)]}$;
	\item[(b)]
$\E[\hat F(\l)]<\infty
	\Longrightarrow v_Y(\l)=-\frac{1}{\E[\hat F(\l)]}$;
	\item[(c)]
$\E[\hat S(\l)]=\infty \mbox{ and }\E[\hat F(\l)]=\infty
	\Longrightarrow v_Y(\l)=0\,.$
\end{itemize}
\end{Proposition}
The proof of Proposition \ref{ariccia}  \ale{has some intersection with the one for the discrete  time case, see \cite[Theorem 2.1.9]{Z}.  The main difference is related to the continuous time   version of  \cite[Lemma 2.1.17]{Z}, since  new  phenomena  have to be controlled. For completeness,  the proof of Proposition \ref{ariccia}   is given in Appendix \ref{nonlaso}}.

For the computation of $\E[\hat S(\l)]$ and $\E[\hat F(\l)]$ we have the following fact (which can be easily verified):
\begin{Lemma}\label{porchetta}
It holds:
\begin{align}
& \E[\hat S(\l)]
		= \E\Big[\frac{\e^{-\l}}{r_0^+}\big(1+\sum_{i=1}^\infty\rho_1\rho_2\cdots\rho_i\e^{-2\l i}\big)\Big]\,, \label{argento1}\\
& \E[\hat F(\l)]
	= \E\Big[\frac{\e^{\l}}{r_0^-}\big(1+\sum_{i=1}^\infty\rho_{-1}^{-1}\rho_{-2}^{-1}\cdots\rho^{-1}_{-i}\e^{2\l i}\big)\Big]\,. \label{argento2}
\end{align}
\end{Lemma}

Similarly to the discrete time case we introduce the thresholds $\l_-$ and $\l_+$ as 
\[  \l_-:= \sup \{ \l\in \bbR  \,:\, v_Y(\l) <0\}\,, \qquad \l_+:= \inf \{ \l \in \bbR \,:\, v_Y(\l) >0\}\,,\]
with the convention that  \ale{$\sup \emptyset =-\infty$ and $\inf \emptyset =+\infty$. }
%
 
\smallskip


Having  Proposition ~\ref{ariccia} and Lemma \ref{porchetta}, by exactly the same arguments used to derive Proposition \ref{marte1} we have:

\begin{Proposition}\label{marte10}  
 The velocity 
$v_Y(\l) $ is   increasing in $\l$ and $\l_-\leq \l_+$. Moreover,    $v_Y(\l)$ is strictly increasing and analytic on $(-\infty , \l_-)$ and on $(\l_+, +\infty)$, while it is zero on $(\l_-, \l_+)$.
 \end{Proposition}
 \begin{Remark}\label{marte20} Due to the above proposition,  $v_Y(\l)$ is analytic everywhere with possible exception at $\l_-, \l_+$, where it can be irregular (even discontinuous, see Section \ref{sec_posate}). 
 \end{Remark}

By the same arguments used to prove Proposition  \ref{tazzina} one easily  gets the following:
\begin{Proposition} \label{tazzinabis} Assume that $\bbE\bigl[ 1/r_0^\pm \bigr]<+\infty$,  $\bbE\bigl[ \log r^\pm _0\bigr]<+\infty$ and 
  $\l_-=\l_+$. Then    $v_Y(\l)$ is continuous, $\l_-=\l_+=\E[\log\rho_0]$ and $v_Y(\l_\pm)=0$.
\end{Proposition}

In the continuous time setting, the RCM has always a regular asymptotic velocity $v_Y(\l)$:
\begin{Example}\label{antonio}
For the \ale{continuous time } RCM  \ale{satisfying our main assumption (see Remark \ref{fieno})} it holds
\begin{align}\label{lento}
v_Y(\l)
	=\frac{\e^{\l}-\e^{-\l}}{\E[1/c_0]}\,.
\end{align}
In particular $v_Y\equiv 0$ iff $\E[1/c_0]=+\infty$. Moreover, $v_Y$ is always an analytic function of $\l$.
\end{Example}
\begin{proof}[Discussion of Example \ref{antonio}]
Since $\rho_i = \frac{c_{i-1}}{c_i}$  from \eqref{argento1} we have
\begin{equation}
\begin{split}
\E[\hat S(\l)]& = \E\Big[\frac{\e^{-\l}}{c_0}\Big(1+\sum_{i=1}^\infty\frac{c_0}{c_1}\frac{c_1}{c_2}\cdots\frac{c_{i-1}}{c_i}\e^{-2\l i}\Big)\Big]=
 \E\Big[\frac{\e^{-\l}}{c_0} +   \sum_{i=1}^\infty \frac{\e^{-\l}}{c_i }   \e^{-2\l i}\Big]\\
& 
	= \E\Big[\frac1{c_0}\Big]\e^{-\l}\sum_{i=0}^{\infty}\e^{-2\l i}
	\,,
\end{split}
\end{equation} thus implying that 
\begin{equation}\label{panino1}
\E[\hat S(\l)] =
\begin{cases}
\E\Big[\frac1{c_0}\Big]\frac{1}{\e^\l-\e^{-\l}} & \l >0\,,\\
+ \infty & \l \leq 0 \,.
\end{cases}
\end{equation}
Similarly   from \eqref{argento2} we have 
\begin{equation*}
\begin{split}
\E[\hat F(\l)]& = 
\ale{
\E\Big[\frac{\e^{\l}}{c_{-1}}\Big(1+\sum_{i=1}^\infty\frac{c_{-1}}{c_{-2} }\frac{c_{-2}}{c_{-3} }\cdots\frac{c_{-i}}{c_{-i-1}}\e^{2\l i}\Big)\Big]
}
=
 \E\Big[\frac{\e^{\l}}{c_{-1}} +   \sum_{i=1}^\infty \frac{\ale{\e^{\l}}}{ \ale{c_{-i-1}} }   \e^{2\l i}\Big]\\
& 
	= \E\Big[\ale{\frac1{c_{0}}}\Big]\e^{\l}\sum_{i=0}^{\infty}\e^{2\l i}
	\,,
\end{split}
\end{equation*}
thus implying that 
\begin{equation}\label{panino2}
\E[\hat F(\l)] =
\begin{cases}
\E\Big[\frac1{c_0}\Big]\frac{1}{\e^{-\l}-\e^{\l}} & \l <0\,,\\
+ \infty & \l \geq  0 \,.
\end{cases}
\end{equation}
From 
\eqref{panino1}, \eqref{panino2} and Proposition \ref{ariccia} we get \eqref{lento}. The conclusion then follows from \eqref{lento}. \end{proof}

\subsection{Reflection invariance}\label{riflessionec}
In the case of reflection invariant environments (cf.~Assumption \ref{jesus_cont}) we prove an analogous of Proposition \ref{fiumicino}:
 
 \begin{Proposition}\label{redenzione} 
Suppose Assumption \ref{jesus_cont} to be satisfied. 
Then it holds
\begin{align}\label{flauto} 
v_Y(\l)=-v_Y(-\l)\,. 
\end{align}
In particular, $v_Y(0)=0$ and, if $v_Y$ has $n$-th derivative at $0$ with $n$ even, then  this derivative must be $0$.
Moreover, the following dichotomy holds for $\l\geq 0$:
\begin{align}\label{gioventu}
\begin{cases}
	\E[\hat S(\l)]<+\infty&\Longrightarrow 
	\quad v_Y(\l)={1}\Big/{\E\Big[\frac{\e^{-\l}}{r_0^+}\big(1+\sum_{i=1}^\infty\rho_1\rho_2\cdots\rho_i\e^{-2\l i}\big)\Big]}\in(0,+\infty)\\
	\E[\hat S(\l)]=+\infty &\Longrightarrow \quad v_Y(\l)=0 \,.
\end{cases}
\end{align}

\end{Proposition}
\begin{proof}
By Assumption \ale{\ref{jesus_cont},} \eqref{argento1} and \eqref{argento2} we have that \begin{align}\label{pongo}
\E[\hat F(\l)]
	&=\E\Big[\frac{\e^{\l}}{r_0^-}\big(1+\sum_{i=1}^\infty\rho_{-1}^{-1}\rho_{-2}^{-1}\cdots\rho_{-i}^{-1}\e^{2\l i}\big)\Big]\nonumber\\
	&=\E\Big[\frac{\e^{\l}}{r_0^+}\big(1+\sum_{i=1}^\infty\rho_1\rho_2\cdots\rho_i\e^{2\l i}\big)\Big]
	=\E[\hat S(-\l)]\,,
\end{align}
which implies $v_Y(\l)=-v_Y(-\l)$ by Proposition \ref{ariccia}. The second property, 
\ale{concerning the derivatives of $v_Y$, follows from \eqref{flauto}. Finally, to  prove the dichotomy, } 
it is enough to check that $\E[\hat F(\l)]=\infty$ whenever $\l\geq 0$.
From \eqref{pongo} we clearly see that, for $\l\geq 0$, $\E[\hat F(\l)]=\E\Big[\frac{\e^{\l}}{r_0^+}\big(1+\sum_{i=1}^\infty\rho_1\rho_2\cdots\rho_i\e^{2\l i}\big)\Big]\geq \E[\hat S(\l)]$. This implies that, whenever $\l\geq 0$ and $\E[\hat S(\l)]=\infty$, also $\E[\hat F(\l)]=\infty$. If on the other hand we have $\l\geq 0$ and $\E[\hat S(\l)]<\infty$, then again we must have $\E[\hat F(\l)]=\infty$. Otherwise we would have, by parts (a) and (b) of Proposition \ref{ariccia}, $v_Y(\l)=\E[\hat S(\l)]^{-1}$ and at the same time $v_Y(\l)=-\E[\hat F(\l)]^{-1}$, which is a contradiction. 
\end{proof}

\subsection{Examples of models with irregular asymptotic velocity $v_Y(\l)$}\label{sec_posate}

\begin{Example}\label{lorenzo3}
Recall Example \ref{lorenzo1} and in particular the random variables $r(x,x\pm 1)$ introduced there. 
Consider the continuous time random walk with jump rates $r_x^\pm:=r(x,x\pm 1)$. Then its asymptotic velocity $v_Y(\l)$ is discontinuous at $\l_+$.
\end{Example}
The proof of the above statement follows from the same arguments presented in Section \ref{lorenzo1_esteso} and therefore is omitted.

\medskip

\begin{Example}\label{lorenzo4} The following model is reflection invariant and its asymptotic velocity $v_Y(\l)$ has no second derivative at $\l_+=\l_-=0$. 
The  jump rates 
 are given by the following. First of all we sample two independent sequences of  positive   i.i.d.~  r.v.'s $(a_m^+)_{m\in\Z}$ and $(a_m^-)_{m\in\Z}$. We call $A:=\E[a_0^+]$ and $B:=\E[1/a_0^+]$  and suppose $A$ and $B$ to be finite.
Then we toss a fair coin and do the following  (see Figure \ref{pocoyo1000}):
\begin{itemize}
	\item If it comes Heads, for all $m\in\Z$ we put 
	\begin{align*}
		\begin{cases}
		r_{2m+1}^+&:=a_m^+=r_{2m+2}^-\,, \\
		r_{2m+1}^-&:=a_m^-=r_{2m+2}^+\,;
		\end{cases}
	\end{align*}
	\item If it comes Tails, for all $m\in\Z$ we put 
	\begin{align*}
		\begin{cases}
		r_{2m}^+&:=a_m^+=r_{2m+1}^-\,,\\
		\ale{r_{2m}^-}&:=a_m^-=r_{2m+1}^+\,.
		\end{cases}
	\end{align*}
\end{itemize}
\end{Example}
The above Example \ref{lorenzo4} is discussed in Section \ref{lorenzo4_esteso}.


\begin{figure}[!ht]
  \centering
\begin{center}
	\centering
	\mbox{\hbox{\includegraphics[width=0.7\textwidth]{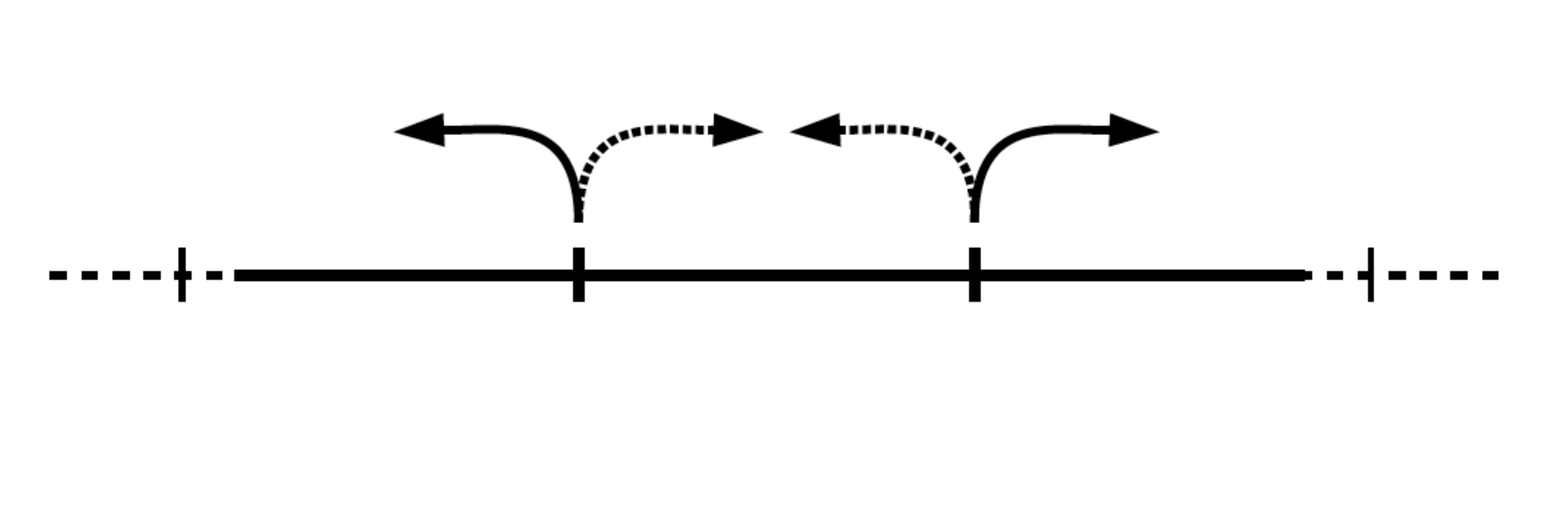}}}
	\end{center}
  \setlength{\unitlength}{0.12\textwidth}  
	\begin{picture}(0,0)
		\put(7.6,0.5){$\tau_k^+$}
		\put(-1.25,1.95){{$a_m^-$}}	 
		\put(0.9,1.95){{$a_m^-$}}	 
		\put(0.22,1.95){{$a_m^+$}}	
		\put(-0.55,1.95){{$a_m^+$}}
		\put(0.4,0.75){{$2m+2$}}	
		\put(-1.1,0.75){{$2m+1$}}  
		\put(2,0.75){\textit{Heads}  }
		\put(0.4,0.3){{$2m+1$}}	
		\put(-0.9,0.3){{$2m$}}  
		\put(2,0.3){\textit{Tails}  }
	\end{picture}
	\caption{Jump rates for Example \ref{lorenzo4}}\label{pocoyo1000}
\end{figure}
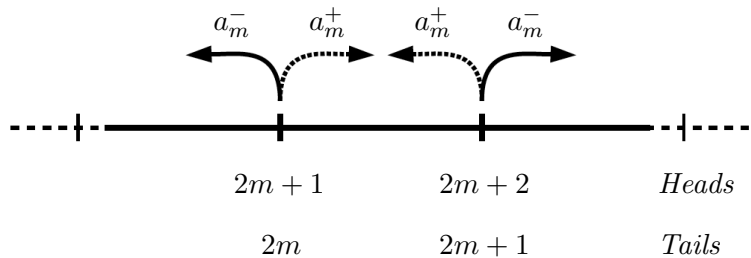


\section{Einstein Relation}\label{sectioneinstein}

In this section we give the proof of the Einstein relation for a discrete time random walk among random conductances. 
 We point out that any 1d nearest-neighbor random walk in random environment, for which  the process \emph{environment viewed from the walker}  admits a reversible distribution, is indeed a random walk among random conductances.

For the continuous time case the Einstein relation follows easily from the expression \eqref{lento} of $v_Y(\l)$, since  the diffusion coefficient of the unbiased random walk is given by   $\s^2(0)= 2/\bbE[ 1/c_0] $.
 We recall that, assuming $\bbE[c_0]<\infty$,  one can prove   an annealed invariance principle with diffusion coefficient $2/\bbE[ 1/c_0] $ (cf. \cite[Eq. (4.22)]{demasi} and references therein), while a quenched invariance principle can be proved by means of the corrector 
under the additional assumption $\bbE[1/c_0]<\infty$. 

For what concerns the discrete time random walk among random conductances we recall that, 
assuming $A:=\bbE[c_0]$ and  $B:=\bbE[1/c_0]$ to be finite,  in the unbiased case (i.e. when $\l=0$)  an annealed (also quenched)  CLT  holds with diffusion coefficient $\s^2(0)=1/AB$ (see   \cite[Eq. (4.20), (4.22)]{demasi} and  \cite[Eq. (4.20) and Exercise 3.12]{B11}).

We prove the  Einstein relation under two different sets of assumptions: the first one requires the ergodicity of the conductances and some moments conditions; we point out that this is the equivalent set of conditions that \cite{LD16} would require in our setting, but we give an alternative, shorter proof. We extend this result including \ale{a} different hypothesis, just requiring the weakest possible integrability of the conductances and very mild mixing conditions.
\begin{Theorem}\label{einstein}
\ale{Consider  a random sequence of conductances $(c_x)_{x\in\Z}$, stationary and ergodic w.r.t. shifts.}
Suppose \ale{that   at least} one of the following two conditions holds:
\begin{itemize}
\item[(i)] \ale{$c_0\in L^p(\P)$ and $1/c_0\in L^q(\P)$ with  $p,q\in [1,\infty]$, $p^{-1}+q^{-1}=1$};
\item[(ii)] \ale{$\lim_{i\to\infty}\E[c_0/c_i]=\E[c_0]\E[1/c_0]$,
with $\E[c_0]<\infty$ and $\E[1/c_0]<\infty$.}
\end{itemize}
Then  \ale{the discrete time RCM satisfies the Einstein relation: i.e. $\frac{d }{d\l}  v_X(\l)  |_{\l=0}=\sigma^2(0)$.
}
\end{Theorem}

\begin{proof}  \ale{Due to Remark \ref{fieno}, our main assumption (see Assumption \ref{uffina}) is satisfied. We prove that $\partial^+_\l v_X(0)=\sigma^2(0)$ (by similar arguments one can prove the same identity for the left derivative)}.
By standard methods   we know that $\sigma^2(0)=\E[c_0]^{-1}\E[1/c_0]^{-1}$ (see   \cite[Eq. (4.20), (4.22)]{demasi} and  \cite[Eq. (4.20) and Exercise 3.12]{B11}). Furthermore, by \eqref{scacco} it is easy to see that both sets \ale{(i) and (ii) of hypotheses} guarantee $\E[\bar S(\l)]<\infty$ for all $\l>0$. 
\ale{Note that, by translation invariance, $\sum_{i=1}^\infty \E\bigl[ \frac{c_0}{c_i}\bigr]=\lim _{k\to \infty}  \E\bigl[ c_{0}^{-1} (c_{-1}+c_{-2}+\cdots+c_{-k}) \bigr]$. By the LLN and  the monotone convergence theorem,  we get that $\sum_{i=1}^\infty \E\bigl[ \frac{c_0}{c_i}\bigr]=\infty$, i.e.  $\E[\bar S(0)]=\infty$.  
}
Hence, \ale{by Proposition \ref{ferb}},   we get for $\l>0$
\begin{align}\label{precipito}
\frac{v_X(\l)-v_X(0)}{\l}
	=\frac{1}{\l+2\l\sum_{i=1}^\infty \E[c_0/c_i]\e^{-2\l i}}\,.
\end{align}
We must show that letting $\l\to 0$ in \eqref{precipito} gives $\E[c_0]^{-1}\E[1/c_0]^{-1}$. We prove that, in fact, $2\l\sum_{i=1}^\infty \E[c_0/c_i]\e^{-2\l i}\to \E[c_0]\E[1/c_0]$ as $\l\to 0$.


\smallskip

\textit{Case $(i)$:} We let $\psi_\l:=(\e^{2\l}-1)\sum_{i=1}^\infty\frac{1}{c_i}\e^{-2\l i}$. By Birkhoff ergodic theorem with exponential weights  we have that
\begin{align}\label{pesato}
\lim_{\l\to 0}\psi_\l
	\xrightarrow{L^q}\E[{1}/{c_0}]\,.
\end{align}
\mic{This can be seen, for example, by setting $m_n:=(c_1^{-1}+c_2^{-1}+\dots+c_n^{-1})/n$ for each $n=1,2,...$ and rewriting $\psi_\l= (1-\e^{-2\l})^2\sum_{n=1}^\infty n\e^{-2\l(n-1)}m_n$. From this and the fact that $(1-\e^{-2\l})^2\sum_{n=1}^\infty n\e^{-2\l(n-1)}=1$, we can bound
\begin{align}\label{treno}
\big\|\psi_\l-\E[1/c_0]\big\|_{q}
	&\leq(1-\e^{-2\l})^2\sum_{n=1}^\infty n\e^{-2\l(n-1)}\big\|m_n-\E[1/c_0]\big\|_q\,.
\end{align}
Since $m_n\to_{L^q}\E[1/c_0]$, given $\varepsilon>0$ by the classical $L^q$ Birkhoff ergodic theorem, we can choose $N\in\N$ such that $\big\|m_n-\E[1/c_0]\big\|_q<\varepsilon$ for each $n\geq N$. We observe that the contribution from the first $N-1$ terms in the r.h.s.~of \eqref{treno} vanishes as $\l\to0$. Hence, as $\l\to0$, we can bound the r.h.s.~of \eqref{treno} by $o(1)+\varepsilon$. This completes the proof of \eqref{pesato}.
}

We write now $2\l\sum_{i=1}^\infty \E[c_0/c_i]\e^{-2\l i}=\frac{2\l}{\e^{2\l}-1}\E[c_0\psi_\l]$. We finally bound by H\"older's inequality
\begin{align*}
\big|\E[c_0\psi_\l]-\E[c_0]\E[c_0^{-1}]\big|
	= \big|\E[c_0(\psi_\l-\E[c_0^{-1}])]\big|
	\leq \|c_0\|_p \cdot \big\|\psi_\l-\E[c_0^{-1}]\big\|_q
	\xrightarrow{\l\to 0} 0\,,
\end{align*}
where the convergence comes from \eqref{pesato}. This concludes the proof of Case (i).

\medskip

\textit{Case $(ii)$:} For $i\in\N$ we let $\E[c_0/c_i]=\E[c_0]\E[c_0^{-1}]+\varepsilon(i)$. For each fixed $k\geq 1$, we can therefore write
\begin{align*}
2\l\sum_{i=1}^\infty \E[c_0/c_i]\e^{-2\l i}
	&=o(\l)+2\l\sum_{i=k}^\infty \E[c_0]\E[c_0^{-1}]\e^{-2\l i}+2\l\sum_{i=k}^\infty  \varepsilon(i) \e^{-2\l i}\\
	&=o(\l)+ \ale{2 \l} \E[c_0]\E[c_0^{-1}]\frac{\e^{-2\l k}}{1-\e^{-2\l }}+2\l\sum_{i=k}^\infty  \varepsilon(i) \e^{-2\l i}\,.
\end{align*}
\ale{Since 
$| \sum_{i=k}^\infty  \varepsilon(i) \e^{-2\l i} |\leq  \sup_{i\geq k}\{|\varepsilon(i)|\}{\e^{-2\l k}}/{(1-\e^{-2\l })}$,}
sending $\l\to 0$ we see therefore that for each $k\geq 1$
\begin{align*}
\E[c_0]\E[c_0^{-1}]- \sup_{i\geq k}\{|\varepsilon(i)|\}
	\leq \lim_{\l\to 0}2\l\sum_{i=1}^\infty \E[c_0/c_i]\e^{-2\l i}
	\leq \E[c_0]\E[c_0^{-1}]+ \sup_{i\geq k}\{|\varepsilon(i)|\}\,.\end{align*}
Finally we let $k\to\infty$ and,  \ale{since $\varepsilon(i)\to 0$, we conclude that \eqref{precipito} converges to $\E[c_0]^{-1}\E[1/c_0]^{-1}$ as $\l \to \infty$, thus implying the thesis.}
\end{proof}


\section{Central limit theorem for ballistic discrete time random walks}\label{caniefigli}

In this section we consider the discrete time random walk $X_n ^{\o,\l}$ and, when ballistic, we investigate its gaussian fluctuations, i.e. the validity of the CLT. As application of the results presented in this section, we will study the CLT for two special model: the discrete time RCM (cf.  Section \ref{diff_RCM})  and the discrete time random walk with i.i.d.~$\o_i^+$'s (see end of this section).

For simplicity of notation,  we write $v(\l)$  instead of $v_X(\l)$.
\ale{We} know that if $\E[\bar S(\l)]<+\infty$ then $v(\l)= 1/\E[\bar S(\l)] >0$ \ale{(see Proposition \ref{ferb})} and moreover  the environment viewed from the perturbed \ale{discrete time} random walk $X_n^{\o,\l}$ admits a steady state $\Q_\l$ \ale{\cite{Z}}. \ale{In addition, there is } a closed formula for the Radon-Nikodym derivative $d\Q_\l/d\P$ that reads \ale{(see \cite[page 185]{Z})},
\begin{align*}
\frac{d\Q_\l}{d\P}=\frac{\L_\l}{\E[\L_\l]}\,,
\end{align*}
where
\begin{align}\label{lambdone}
\L_\l
	:=\frac{1}{\o_0^+(\l)}\Big(1+\sum_{i=1}^\infty\prod_{j=1}^i\rho_j(\l)\Big)
	=(1+\rho_0\e^{-2\l})\Big(1+\sum_{i=1}^\infty\rho_1\rho_2\cdots\rho_i\e^{-2\l i}\Big)\,.
\end{align}
It is simple to prove that $\E[\L_\l]=\E[\bar S(\l)]=1/v(\l)$. 

For $n\in\Z$, we introduce also the $n^{th}$ shift of the function $\bar S(\l)$ as 
\begin{align*}
\theta^n\bar S(\l)
	:=1+2\sum_{i=0}^\infty \rho_n\rho_{n-1}\cdots\rho_{n-i}\,\e^{-2\l (i+1)}\,.
\end{align*}
Equivalently, $\theta^n\bar S(\l)[\o]= \bar S(\l)[ \theta ^n \o]$ \ale{(in the r.h.s. $\theta $ denotes the usual shift on $\Omega$ with $(\theta \o)_x:= \o_{x+1}$)}. 
\begin{Assumption}\label{gatto}
The expectation $\E[\bar S(\l)]$ is finite. Furthermore, there exists $\varepsilon>0$ such that
\begin{align}\label{erica}
\E_{\Q_\l}[\omega_0^+(\l)\bar S(\l)^{2+\ep}+\omega_0^-(\l)\theta^{-1}\bar S(\l)^{2+\ep}]<\infty 
\end{align}
and moreover it holds
\begin{align}\label{soreta}
\sum_{n\geq 1}\sqrt{\E\Big[\E[v(\l)\bar S(\l)-1\,|\,\mathcal G_{ -n}]^2\Big]}<\infty\,,
\end{align}
where, for each $k$, $\mathcal G_{k}$ is any $\sigma$-algebra for which the random variables $(\theta^{i}\bar S)_{i\leq k}$ are measurable.
\end{Assumption}

\begin{Theorem}\label{baldo}
Under Assumption \ref{gatto}, the discrete time random walk in random environment satisfies the annealed CLT
\begin{align*}
\frac{\ale{X_n^{\o,\l}}-v_X(\l)\cdot n}{\sqrt n}
	\longrightarrow \mathcal N(0,\sigma^2(\l))
	\qquad \mbox{under } \P\otimes \ale{P^{\omega,\l}_0}\,.
\end{align*}
The diffusion coefficient is given by
\begin{equation}\label{coeff_diff}
\sigma^2(\l)=\sigma_{1}^2(\l)+v(\l)\sigma_{2}^2(\l)\ale{\,\in\,(0,\infty)}\,,
\end{equation}
where
\begin{equation}\label{criceto1}
\begin{split}
\sigma_{1}^2(\l)
&:= v(\l)^2 \bbQ_\l \Big[
\omega_0^+(\l)\big(\bar S(\l)-1\big)^2+\omega_0^-(\l)\big(\theta^{-1}\bar S(\l)+1\big)^2
\Big] \\	
&=v(\l)^3\E\Big[\Lambda_\l\omega_0^+(\l)\big(\bar S(\l)-1\big)^2+\Lambda_\l\omega_0^-(\l)\big(\theta^{-1}\bar S(\l)+1\big)^2\Big]\,,
	\end{split}
\end{equation}
and
\begin{equation}\label{criceto2}
\sigma_{2}^2(\l)
	\ale{:=}\E\Big[\big(v(\l)\bar S(\l)-1\big)^2\Big]+2\sum_{n\geq 1}\E\Big[\big(v(\l)	\bar S(\l)-1\big)\big(v(\l)\theta^{n}\bar S(\l)-1\big)\Big]\,.
\end{equation}
The series in \eqref{criceto2} is absolutely convergent.
Furthermore, $\sigma^2(\l)$ is analytic on every open interval given by values of $\l$ that satisfy Assumption \ref{gatto}.
\end{Theorem}
\mic{The proof of the above theorem is postponed to   Appendix \ref{spagna}.
}

\smallskip

Theorem \ref{baldo} is an extension in our context of Theorem 2.2.1 in \cite{Z}\footnote{In the definition of $\s^2_{P,1}$ given in \cite[Theorem 2.2.1]{Z} there is a typo. One should replace $Q$ by $\bar Q$, otherwise the centered formula above (2.2.8) in \cite{Z} fails.}. In fact, condition \eqref{erica} is replaced therein by the stronger condition
\begin{align}\label{erica2}
\E_{\Q_\l}[\bar S(\l)^{2+\ep}+\theta^{-1}\bar S(\l)^{2+\ep}]<\infty
\end{align}
and \eqref{soreta} is stated there with $\mathcal G_{-n}$ given by the $\sigma$-algebra generated by $\omega_i^\pm$ with $i\leq -n$.
We point out that condition \eqref{erica2} is in general not optimal. For example in the random conductances case \eqref{erica} is much less restrictive. The fact that $\mathcal G_{-n}$ could be taken more general than in  \cite[Theorem 2.2.1]{Z} is implicit in the proof provided in \cite{Z}, but as it will be clear from our examples below it is sometimes more useful to work with different $\sigma$-algebras.

\begin{Remark}
The arguments used in the proof of Theorem \ref{baldo} that show how to replace condition \eqref{erica2} by \eqref{erica} are general. In particular, Theorem 2.2.1 in \cite{Z} can be proved substituting assumption (2.2.2) therein by
$E_Q(\omega_0^+\bar S^{2+\ep}(\omega)+\omega_0^-\bar S(\theta^{-1}\omega)^{2+\ep})<\infty $, which is the equivalent of \eqref{erica} in the general setting. Here we have used the notation of \cite{Z}.
\end{Remark}


\mic{We provide now} alternative formulas for \eqref{criceto1} and \eqref{criceto2} and give a sufficient condition for the analyticity of $\sigma^2(\l)$. In the next subsection we also give some comments on Assumption \ref{gatto}.

\begin{Proposition}\label{sviluppo}
Set 
\begin{align*}
	 U=U(\omega,\l):= \sum_{i=0}^\infty \rho_0\dots\rho_{-i}\e^{-2\l(i+1)} 
	\qquad\mbox{and}\qquad
V=V(\omega,\l):=\sum_{i=1}^\infty \rho_1\dots\rho_i\e^{-2\l i}\,.
\end{align*}
Then 
\begin{equation}\label{coldiretti1}
\sigma_{1}^2(\l)
	=\frac{4 \E[ U^2+V+2VU^2+2VU]}{(1+2\E[U])^3}   \,,
	\end{equation}
and 
\begin{equation}\label{coldiretti2}
\begin{split}	
v(\l) \sigma_{2}^2(\l)
	&=\frac{4}{(1+2\E[U])^3}	\Big[ \E[U^2]-\E[U]^2 +2\sum_{n\geq 1} \Big(  \E[U \theta^n U] - \E[U]^2\Big)  \Big]
		\,.\end{split}
\end{equation}
\ale{The  series in  \eqref{coldiretti2} is absolutely convergent.}
\end{Proposition}

\begin{Remark}
\ale{Identity} \eqref{coldiretti2} can  also be  written as 
\begin{align*}
v(\l) \sigma_{2}^2(\l)
	=\frac{4}{(1+2\E[U])^3}\sum_{n\in \bbZ} {\rm Cov}( U, \theta ^n U)\,.
\end{align*}
\end{Remark}

%
%
%

\begin{proof}[Proof of Proposition \ref{sviluppo}]
We note that  $\bar S(\l)=1+2U$ and $\Lambda_\l=\frac{1}{\omega_0^+(\l)}(1+V)$, 
and start looking at $\sigma_{1}^2(\l)$. For each $n\in\Z$, $\theta^nU(\omega, \l)$ and $\theta^nV(\omega, \l)$ will just indicate respectively $U(\theta^n\omega,\l)$ and $V(\theta^n\omega,\l)$.
We have 
\begin{align}\label{ercole1}
\Lambda_\l\omega_0^+(\l)\big(\bar S(\l)-1\big)^2
	=4(1+V)U^2
\end{align}
and 
\begin{align}
\Lambda_\l\omega_0^-(\l)\big(\theta^{-1}\bar S(\l)+1\big)^2
	&=\rho_0(\l)\Big(1+\sum_{i=1}^\infty \rho_1\dots\rho_i\e^{-2\l i}\Big)
	\Big( 2+2 \theta ^{-1} U \Big)^2 \nonumber\\
	&=4\Big(\sum_{i=0}^\infty \rho_0\dots\rho_i\e^{-2\l (i+1)}\Big)
	\Big( 1+ \theta ^{-1} U \Big)^2= 4 \theta^{-1} V \Big( 1+ \theta ^{-1} U \Big)^2\,.
\end{align}
Taking the expectation and using the translation invariance of $\bbP$ we see then that
\begin{align}\label{ercole2}
\E\Big[\Lambda_\l\omega_0^-(\l)\big(\theta^{-1}\bar S(\l)+1\big)^2\Big]
	&=4\E[V(1+U)^2]\,.
\end{align}
Putting together \eqref{ercole1} and \eqref{ercole2} back into \eqref{criceto1} we obtain \eqref{coldiretti1}.

We move to $\sigma_{2}^2(\l)$ and notice that the first addendum in \eqref{criceto2} is equal to ${4}\E[(U-\E[U])^2]/(1+2\E[U])^2$.
For second addendum in \eqref{criceto2} we notice that 
\begin{align*}
\E\Big[\big(v(\l)\bar S(\l)-1\big)\big(v(\l)\theta^{n}\bar S(\l)-1\big)\Big]
	&=v(\l)^2\E\Big[\big(\bar S(\l)-\E[\bar S(\l)]\big)\big(\theta^{n}\bar S(\l)-\E[\bar S(\l)]\big)\Big]\\
	&=\frac{4\E\big[\big(U-\E[U]\big)\big(
	\theta^nU-\E[U]\big)\big]}{(1+2\E[U])^2}\,,
\end{align*}
thus implying \eqref{coldiretti2}.
\end{proof}

\begin{Proposition}\label{analiticita}
Let Assumption \ref{gatto} be satisfied with $\l$ in a given open interval $I$. Given $n\geq 1$ and $k\geq 0$ set
\begin{align}\label{aennecappa}
a_{n,k}
	:=\sum_{i,j\geq 0:\,i+j=k}{\rm Cov}(\rho_0\cdots\rho_{-i},\rho_n\cdots\rho_{n-j})
\end{align}
and suppose that
\begin{align}\label{ritorno}
\sum_{k=0}^\infty \e^{-2\l k}\Big(\sum_{n=1}^\infty |a_{n,k}|\Big)<\infty\qquad \forall \l\in I\,.
\end{align}
Then the coefficient $\sigma^2(\l)$ is analytic in $I$.
\end{Proposition}
\begin{proof}
To get the thesis we apply Lemma \ref{latticini} below to prove the analyticity of $\s_1^2(\l)$ and of $v_X(\l) \s_2^2(\l)$ separately. \ale{By Proposition \ref{ferb} and \ref{marte1}} we already know that $v_X(\l)$, and therefore $1/(1+2\E[U])^3$, is analytic and positive. Hence, to prove the analyticity of $\s_1^2(\l)$  it is enough to prove the analyticity of the numerator in the right term in \eqref{coldiretti1}.
We prove the analyticity of $\E[VU^2]$ by Lemma \ref{latticini} (the other contributions in the numerator can be treated similarly).
We set 
\[ a_n:=\sum_{\substack{(i,j,k)\in \N_+^3:\\i+j+k=n}}\E\bigl[( \rho_0\rho_1\cdots \rho_{-i+1} )
\ale{( \rho_0\rho_1\cdots \rho_{-j+1} )}
(\rho_1\rho_2\cdots \rho_k) \bigr]\,.
\]
Then $\E[VU^2]= \sum _{n\geq 0} a_n \e^{ - 2 \l n}$ (note that $a_n \geq 0$, so terms can be ordered arbitrarily in the involved series).
Since we know by the CLT that $\s_1^2(\l)$ is well defined and finite for $\l \in I$, by \eqref{coldiretti1}  it must be 
$\E[VU^2]<\infty$ for any $ \l \in I$, and therefore $ \sum _{n\geq 0} a_n \e^{ - 2 \l n}<\infty$ for any $\l\in I$. By Lemma \ref{latticini} we conclude that  $\E[VU^2]$ is analytic on $I$.

\smallskip

We move to  $v_X(\l) \s_2^2(\l)$ as expressed in  \eqref{coldiretti2}. By applying the same arguments as above we reduce to prove the analyticity on $I$ of $\sum_{n\geq 1}(\bbE[ U\cdot  \theta^n U]-\bbE[U]^2 )$. 
We write $F_n(\l):=\bbE[ U\cdot  \theta^n U]-\bbE[U]^2$. By hypothesis \eqref{ritorno} we have that $F_n(\l)=\sum_{k=0}^\infty \e^{-2\l (k+2)}a_{n,k}$ and this series is absolutely convergent. Summing over $n$ and using again hypothesis \eqref{ritorno} we have
\begin{align*}
\sum_{n\geq 1}F_n(\l)=\sum_{k\geq 0} \e^{-2\l (k+2)} b_k\,,
\end{align*}
where $b_k:=\sum_{n\geq 1}a_{n,k}$ and the series in the right hand side of the last display is absolutely convergent.
By Lemma \ref{latticini} this series is therefore an analytic function of $\l$ in $I$.
\end{proof}

\subsection{Sufficient conditions for Assumption \ref{gatto}}\label{trilli}
\begin{Lemma}\label{brusco}
Assume $\E[\bar S(\l)]<\infty$. Then \eqref{erica} is satisfied if and only if both $\E[U^{2+\ep}]<\infty$ and $\E[VU^{2+\ep}]<\infty$.
\end{Lemma}
\begin{proof}
By the same \ale{computations} appearing in the proof of Proposition \ref{sviluppo} we have
\begin{align}
& \Lambda_\l\ale{\omega_0^+(\l)}\bar S(\l)^{2+\ep}=(1+V)(1+2U)^{2+\ep}\,,\\
& \Lambda_\l\ale{\omega_0^-(\l)} \theta^{-1}\bar S(\l)^{2+\ep}=\ale{(\theta^{-1}V)}(1+2\theta^{-1}U)^{2+\ep}\,.
\end{align}
Hence, using translation invariance, we see that \ale{\eqref{erica}} is satisfied as long as $\E[U^{2+\ep}]<\infty$ and $\E[VU^{2+\ep}]<\infty$
(notice that $\E[V]<\infty$ since $\E[\bar S(\l)]<\infty$ by hypothesis). 
\end{proof}

\begin{Proposition}\label{alternativa1}
Assume $\E[\bar S(\l)]<\infty$. Given $\varepsilon>0$ suppose that there exists $\delta>0$ such that
\begin{align}
&\sum_{i=0}^\infty \E[\rho_0^{2+\ep}\cdots\rho_{-i}^{2+\ep}]\,\e^{-2\l i (2+\ep-\delta)}<\infty\,, \label{freddo1}\\
&\sum_{i=0}^\infty\sum_{j=1}^\infty \E[\rho_0^{2+\ep}\cdots\rho_{-i}^{2+\ep}\rho_1\cdots\rho_j]\,\e^{-2\l i (2+\ep-\delta)-2\l j}<\infty \label{freddo2}\,.
\end{align}
Then \eqref{erica} is satisfied.
\end{Proposition}

\begin{proof}
We apply Lemma \ref{brusco}. \ale{We show that \eqref{freddo1} implies that $\E[U^{2+\ep}]<\infty$ and   that \eqref{freddo2} implies $\E[VU^{2+\ep}]<\infty$}.
We define the measure $\mu$ on $\N$ as $\mu(i):=\e^{-2\l\alpha i}$, with $\alpha>0$ to be determined later. We write $U=\mu(f)$, with $f(i)=\rho_0\dots\rho_{-i}\,\e^{-2\l(i+1)+2\l \alpha i}$. By \ale{H\"older's} inequality we have $U=\mu(f)\leq \|f\|_{L^{2+\ep}(\mu)}\|1\|_{L^q(\mu)}=C\,\|f\|_{L^{2+\ep}(\mu)}$, where $q$ is the conjugate exponent of $2+\ep$ and $C>0$ is a finite  constant. As a consequence
\begin{equation}\label{andre}
\ale{ U^{2+\ep}
	\leq C^{2+\ep}  \mu(f^{2+\ep})
	=C^{2+\ep} \e^{-2\l (2+\ep)}\sum_{i=0}^\infty   \rho_0^{2+\ep}\cdots\rho_{-i}^{2+\ep}\,\e^{-2\l i (2+\ep-\alpha (1+\ep))}.}
\end{equation}
\ale{To prove \eqref{freddo1} it is enough to  choose  $\alpha=\delta/(1+\ep)$ and take the expectation w.r.t. $\bbP$. To prove \eqref{freddo2} it is enough to multiply both sides of \eqref{andre} by $V$ and conclude as just   done for \eqref{freddo1}.}
\end{proof}

We also give some comments on \eqref{soreta}. First of all we notice that, by the definition of $U$, \eqref{soreta} is equivalent to asking
\begin{align}\label{frateto}
\sum_{n\geq 1}\sqrt{\E\Big[\E[U-\E[U]\,|\,\mathcal G_{ -n}]^2\Big]}<\infty\,.
\end{align}
We give sufficient conditions for this to hold.
\begin{Proposition}\label{alternativa2}
Suppose that $\E[U^2]<\infty$.
Assume that there exists $M>0$ such that, for all $i\in\Z$, $\rho_i$ is independent from all $\rho_j$'s with $|i-j|> M$. Further suppose that
\begin{align}\label{freddo3}
\sum_{i=0}^\infty \E[\rho_0^{2}\cdots\rho_{-i}^{2}]^{1/2}\,i\,\e^{-2\l i }
	<\infty\,.
\end{align}
Then \eqref{soreta} is satisfied with $\mathcal G_k$ being the $\sigma$-algebra generated by all the $\rho_i$ with $i\leq k$.
\end{Proposition}
We point out that trivially $\theta^k\bar S(\l)$ is $\mathcal G_k$ measurable.
\begin{proof}
We will check \eqref{frateto}. Since $\E[U^2]<\infty$ by hypothesis, all terms in the series \ale{\eqref{frateto}} are finite. Therefore, without loss, we can ignore the first terms in \eqref{frateto} and prove that
\begin{align}\label{babbasone}
\sum_{n= M}^\infty\big\| \E[U-\E[U]\,|\,\mathcal G_{ -n}]  \big\|_{L^2(\P)}<\infty\,.
\end{align}
Using the finite-range dependence assumption, we can bound  
\begin{align*}
\E[U-\E[U]\,|\,\mathcal G_{ -n}]
	&=\sum_{i=0}^{\infty} \big( \E[\rho_0\cdots\rho_{-i}\,|\, \cG_{-n}]- \E[\rho_0\cdots\rho_{-i}]\big)\e^{-2\l (i+1)}
		\\
	&=\sum_{i= n-M}^{\infty} \big( \E[\rho_0\cdots\rho_{-i}\,|\, \cG_{-n}]- \E[\rho_0\cdots\rho_{-i}]\big)\e^{-2\l (i+1)}\,.
\end{align*}
Hence the l.h.s.~of \eqref{babbasone} can be bounded by
\begin{align*}
\sum_{n= M}^\infty&\sum_{i= n-M}^{\infty} \big\| \E[\rho_0\cdots\rho_{-i}\,|\, \cG_{-n}]- \E[\rho_0\cdots\rho_{-i}]\big\|_{L^2(\P)}\e^{-2\l (i+1)}\\
	&\leq \sum_{n= M}^\infty\sum_{i= n-M}^{\infty} 2 \E[\rho_0^2\cdots\rho_{-i}^2]^{1/2}\e^{-2\l (i+1)}\,,
\end{align*}
 which allows to conclude.
\end{proof}

We conclude this section with an application of the above results. Another application is given by the analysis of the gaussian fluctuations in the discrete time random conductance model provided in the next section.

\begin{Example}\label{pierpaolo2007} 
Consider the case of    i.i.d.~$\o_i^+$'s  such that $\bbE[ \rho_0^{2+\ep }]<+\infty$. Then  Assumption \ref{uffina} is satisfied. Moreover, for $\l > \frac{1}{2(2+\ep )}\log \bbE[ \rho_0^{2+\ep }]$,  Assumption \ref{gatto} is fulfilled and therefore the content of Theorem \ref{baldo} is valid, and  $\s^2(\cdot)$ is analytic.
 \end{Example}
\begin{proof}[Discussion of Example \ref{pierpaolo2007}] It is trivial to check the validity of Assumption \ref{uffina}.
By Propositions \ref{alternativa1} and  \ref{alternativa2} one can check that Assumption \ref{gatto} is satisfied if $\E[\rho_0^{2+\ep}]<\e^{2\l(2+\ep)}$. In fact, this guarantees, by H\"older inequality, that 
$\E[\rho_0]\leq\E[\rho_0^{2+\ep}]^{1/(2+\ep)}<\e^{2\l}$ and 
$\E[\rho_0^{2}]\leq \E[\rho_0^{2+\ep}]^{2/(2+\ep)}<\e^{4\l}$. Thanks to these three bounds it is easy to verify that $\E[\bar S(\l)]<\infty$, $\E[U^2]<\infty$ and that \eqref{freddo1}, \eqref{freddo2} and \eqref{freddo3} are verified.

We show now that $\sigma^2(\l)$ is analytic for all $\l$ satisfying $\E[\rho_0^{2+\ep}]<\e^{2\l(2+\ep)}$. To this aim we apply Proposition \ref{analiticita}. We start looking at \eqref{aennecappa}. When $n>k$ we have by independence $a_{n,k}=0$. When this is not the case we estimate
\begin{align*}
|{\rm Cov}(\rho_0\cdots\rho_{-i},\rho_n\cdots\rho_{n-j})|
	&\leq {\rm Var}(\rho_0\cdots\rho_{-i})^{1/2}{\rm Var}(\rho_n\cdots\rho_{n-j})^{1/2}\\
	&\leq \E[\rho_0^2\cdots\rho_{-i}^2]^{1/2}\E[\rho_n^2\cdots\rho_{n-j}^2]^{1/2}
	=\E[\rho_0^2]^{(i+j+2)/2}\,.
\end{align*}
For $n\leq k$ we can therefore bound $a_{n,k}\leq (k+1)\E[\rho_0^2]^{(k+2)/2}$. By the above observations we can bound the l.h.s.~ of \eqref{ritorno} by
\begin{align*}
\sum_{k=0}^\infty \e^{-2\l k} k(k+1) \E[\rho_0^2]^{(k+2)/2}\,.
\end{align*}
The last display is finite as soon as $\E[\rho_0^2]<\e^{4\l}$, which we have shown to be true under the condition $\E[\rho_0^{2+\ep}]<\e^{2\l(2+\ep)}$.
\end{proof}

\section{Diffusion coefficient in the discrete time random conductance model}\label{diff_RCM}
We consider here  the discrete time RCM with i.i.d. conductances  and study in detail its gaussian fluctuations. Assuming $A:=\bbE[c_0]$ and  $B:=\bbE[1/c_0]$ to be finite,  in the unbiased case (i.e. when $\l=0$) it is known that $v_X(0)=0$ and  that, under diffusive rescaling, an annealed (also quenched)  CLT  holds with diffusion coefficient $\s^2(0)=1/AB$ (see   \cite[Eq. (4.20), (4.22)]{demasi} and  \cite[Eq. (4.20) and Exercise 3.12]{B11}). We recall that $v_X(\l)$ is given by \eqref{v_RCM} for $\l \geq 0$ and that $v_X(\l)= - v _X(-\l)$ (see Proposition \ref{fiumicino}). By applying the results obtained in Section \ref{caniefigli} we get the following:

\begin{Theorem}\label{suppongo} Consider the discrete time RCM with i.i.d. conductances $(c_i)_{i \in \bbZ}$.
Suppose that for some $\ep>0$ it holds $\E[c_0^{2+\ep}]<\infty$ and $\E[1/c_0^{2+\ep}]<\infty$. Then   Assumption \ref{uffina} is satisfied. Moreover,  for any $ \l \in \bbR$, the discrete time random walk in random environment satisfies the annealed CLT
\begin{align}\label{diluvio}
\frac{ X_n^{\o,\l}-v_X(\l)\cdot n}{\sqrt n}
	\longrightarrow \mathcal N(0,\sigma^2(\l))
	\qquad \mbox{under } \P\otimes P^{\omega,\l}_0\,.
\end{align}
The above  diffusion coefficient $\sigma^2(\l)$ satisfies  $\s^2(0)=\frac{1}{AB}$,   $\s^2(\l)=\s^2(-\l)$ for all $\l\in \bbR$   and it is given, for $\l >0$, by
\begin{equation}\label{rey}
\begin{split}
\sigma^2(\l)
	=\frac{4(\e^{2\l}-1)^2}{(\e^{2\l}-1+2AB)^3}\Big[ & \frac{2 CD}{\e^{2\l}+1}+\frac{4(A^2D+B^2C)}{\e^{4\l}-1}
	+\frac{8A^2B^2}{(\e^{2\l}-1)(\e^{4\l}-1)}\\ &+AB+\frac{4AB-A^2B^2}{\e^{2\l}-1}-\frac{2A^2B^2\e^{2\l}}{(\e^{2\l}-1)^2}\Big]\,,
\end{split}
\end{equation}
where  $A:=\E[c_0]$, $B:=\E[1/c_0]$, $C:=\E[c_0^2]$ and $D:=\E[1/c_0^2]$.
\end{Theorem}
The proof of Theorem \ref{suppongo} follows from the results presented in Section \ref{caniefigli} by straightforward computations and is therefore postponed to Appendix \ref{larome}.

\medskip

In the deterministic case, that is, when the conductances are all equal to a constant, we obtain a biased simple random walk whose diffusion coefficient $ \sigma^2(\l)$ can be easily computed (or, equivalently, derived from \eqref{rey}):  \begin{equation}\label{sigmino}
 \sigma^2(\l) =4 /(\e^{
 \l}+\e^{-\l})^4 \qquad \forall \l \in \bbR\,.
 \end{equation}
 See Figure \ref{pellegrino} for a plot of $\s^2(\l)$. We note that in the deterministic case $\s^2(\l)$ is analytic, it is strictly increasing for $\l <0$ and strictly decreasing for $\l>0$.
The following proposition treats the non deterministic case, in which a different behavior emerges.
\begin{figure}[!ht]
	\begin{center}
	\centering
	\mbox{\hbox{
	\includegraphics[width=0.6\textwidth]{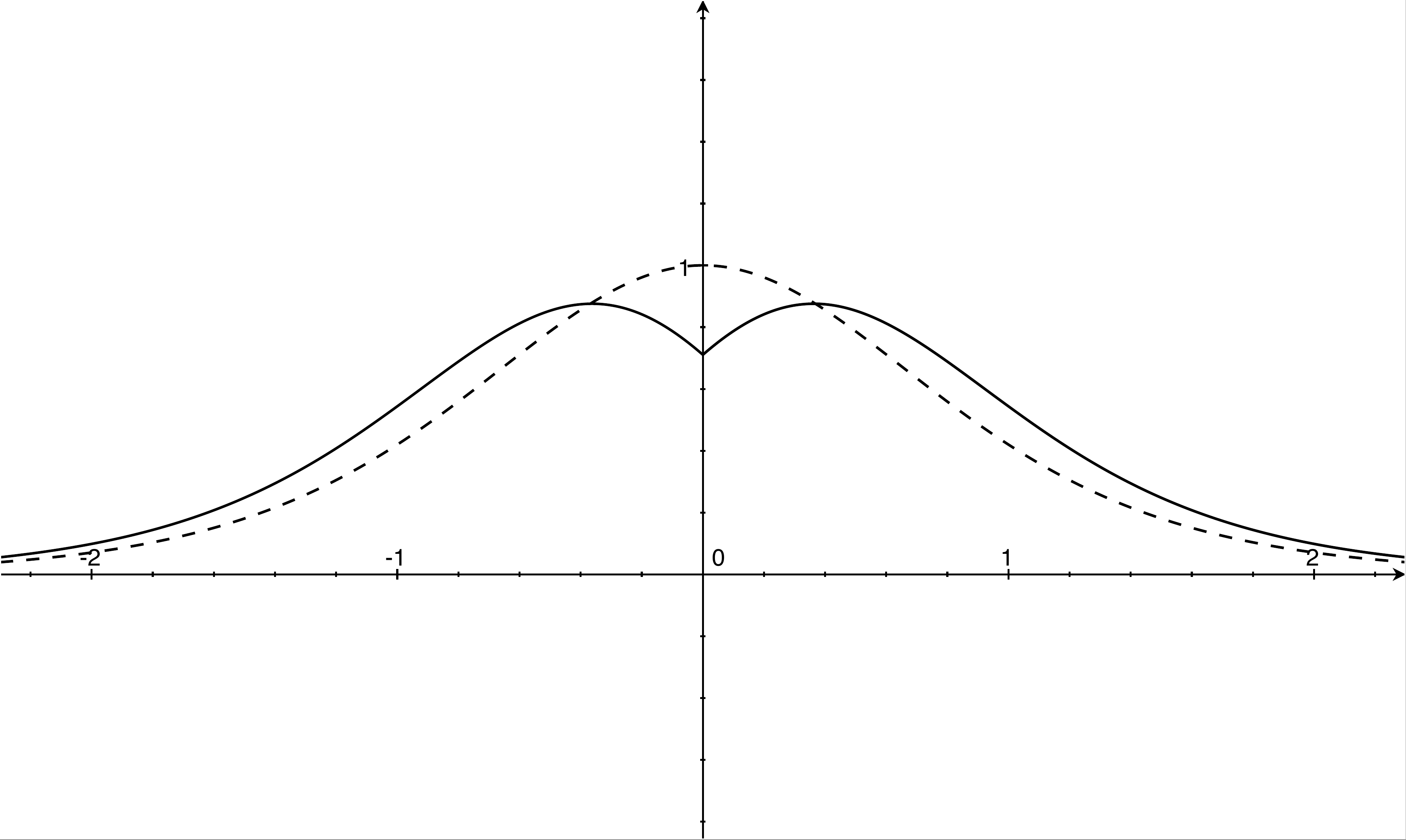}}}
	\end{center}\caption{The solid line represents the function $\sigma^2(\l)$ in the case of i.i.d.~conductances sampled from a uniform random variable in the interval $[1,10]$. This is plotted against the function $\sigma^2(\l)$ (dotted line) in the deterministic case of constant conductances, cfr.~\eqref{sigmino}.}
	\label{pellegrino}
\end{figure}

\medskip

\begin{Proposition}\label{unpofru}
For genuinely random conductances, $\s^2(\l)$ is continuous everywhere, it is analytic on $\bbR\setminus\{0\}$, but it is not in general differentiable at $\l=0$. Moreover, $\s^2(\l)$ is in general not monotone on $(-\infty,0]$ and on   $[0,\infty)$.
\end{Proposition}
\begin{proof}
Continuity on $(0,+\infty)$ follows directly by \eqref{rey}. For small $\l>0$ we can calculate, by a Taylor expansion,
\begin{align*}
\frac{1}{(\e^{2\l}-1+2AB)^{3}}&= 
 \frac{1}{( 2AB+2\l +o(\l))^3}= \frac{1}{8 A^3 B^3} \frac{1}{( 1+ \frac{\l}{AB} +o(\l))^3}=\frac{1- \frac{3\l}{AB} +o(\l) }{8 A^3 B^3}\,.
\end{align*}
Hence, for small $\l>0$  we can rewrite \eqref{rey} as
\begin{equation}\label{fifu}
\begin{split}
\sigma^2(\l)
	 =& \frac{1}{2 A^3B^3} \Big( 1- \frac{3\l}{AB} \Big)
		\Big[  4(A^2D+B^2C)\l +4 A^2B^2(1-\l)\\ 
		&\quad +(4AB-A^2B^2)2\l -2A^2B^2(1+2\l)\Big]+o(\l)\\
	=&\frac{1}{AB}+a_1\l+o(\l)\,,
	\end{split}
\end{equation}
where
\begin{equation}\label{auno}
\begin{split}
 a_1
 := \frac{2( A^2D+B^2C)}{A^3B^3}-\frac{5}{AB}+\frac{1}{A^2B^2}\,.
\end{split}
\end{equation}
Since $\sigma^2(0)$ is equal to $(AB)^{-1}$ (see the discussion before  Theorem \ref{suppongo})  we see that also in the genuinely random case (i.e., when $AB\not=1$) $\sigma^2(\l)$ is a continuous function everywhere.

\smallskip

\begin{figure}[!b]
	\begin{center}
	\centering
	\mbox{\hbox{
	\includegraphics[width=0.5\textwidth]{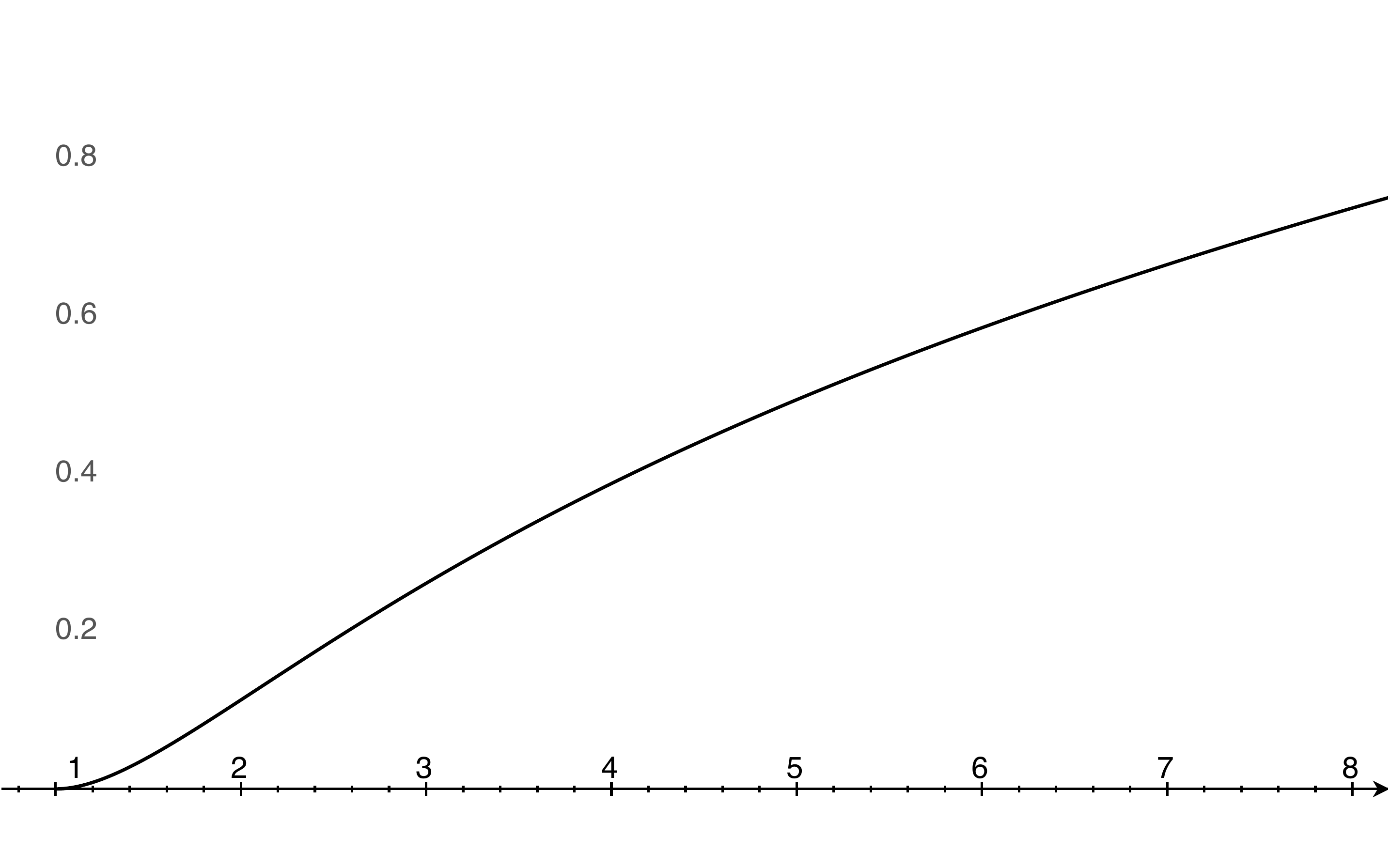}}}
	\end{center}\caption{Plot of $a_1(x)$, the first order coefficient of the Taylor series of $\sigma^2(\l)$, when the i.i.d.~conductances are  uniform random variables between $1$ and $x$.}
	\label{pocoyo100}
\end{figure}

On the other hand, by reflection invariance, we note that if $a_1\not=0$ then $\sigma^2(\l)$ cannot have first derivative in $\l=0$ and in particular cannot be an analytic function. In fact, $\sigma^2(\l)=\sigma^2(-\l)$ implies that the right derivative of $\sigma^2$ in $0$ is equal to the opposite of the left derivative. Hence, if the first derivative in $0$ exists, it has to be equal to $0$, that is, we should have $a_1=0$. It is easy to exhibit examples where this is not true:
\begin{itemize}
\item[(1)] Take i.i.d.~conductances that can assume only two values. Without loss of generality, we suppose they attain value $M>1$ with probability $p\in(0,1)$ and value $1$ with probability $(1-p)$. Multiplying $a_1$ by $A^3B^3>0$ we obtain
\begin{align*}
A^3B^3 a_1
	=&p(6-5M'+2M'')+p^2(-36+25M'-7M'')\\
	&+p^3(60-40M'+10M'')+p^4(-30+20M'-5M'')\,,
\end{align*}
where $M'=M+M^{-1}$ and $M''=M^2+M^{-2}$. It is not hard to show that this number is always strictly bigger than $0$, so that we also always have $a_1>0$.
\item[(2)] Consider independent conductances sampled from a uniform distribution between $1$ and $x>1$. One can easily calculate $A= (1+x)/2$, $B= \log x /(x-1)$, $C=  (x^3-1)/(3(x-1))$ and $D= 1/x$; it follows that
\begin{align*}
a_1(x)
	=\frac 4{x(x+1)}\frac{(x-1)^3}{\log^3 x}+\frac {16}{3}  \frac{x^3-1}{(x+1)^3\log x}
	- \frac{10(x-1)}{(x+1)\log x}+\frac{4(x-1)^2}{(x+1)^2\log^2 x}\,,
\end{align*}
and also in this case $a_1=a_1(x)>0$ as soon as $x>1$ (see Figure \ref{pocoyo100}).
\end{itemize}
These examples, besides representing cases of non-differentiability of $\sigma^2(\l)$ in $\l=0$, also show that $\sigma^2(\l)$ can be strictly increasing for small $\l>0$, in contrast with the deterministic case where we showed $\sigma^2(\l)$ to be always decreasing for $\l>0$ (notice that in every case $\sigma^2(\l)\to 0$ as $\l\to\infty$). 
\end{proof}

It could be possible  that the non-monotonicity of $\sigma^2(\l)$ for $\l\geq 0$ happens as soon as some disorder is introduced in the system. Maybe the coefficient  $a_1$ appearing in \eqref{auno} is strictly bigger than $0$ whenever the conductances are genuinely random.

\section{Discussion of Example \ref{lorenzo1}}\label{lorenzo1_esteso}

We recall that $A>0$ and $\g>2$ (for most of the arguments below one can take $\g>1$ but at the end it is necessary to take $\g>2$).
Moreover, we have  set $r(k,k-1):=A$ for all $k\in\Z$ and we have introduced  $\tilde \tau=(\tilde\tau_k)_{k\in\Z}$, which is  a renewal point process on $\bbZ$  with $\tilde\tau_0=0$ \ale{and,  for $k\not=0$, 
 $\P(\tilde{\tau}_{k+1}- \tilde{\tau}_{k} \geq j)=c/j^\gamma$ for any $j  \in\N_+$. }
 
We write $\t = (\tau_k)_{k\in\Z}$ for the renewal point process  given by the $\bbZ$--stationary version of $\tilde \tau$. More precisely, numbering  the points of $\t$ in increasing order with $\t_0 \leq 0 <\t_ 1$, $\t$ is  characterized as follows: (i) the law of $\t$ is left invariant by integer shifts; (ii) the random variables $(\t_{i+1}-\t_i: i\in \bbZ)$ are independent and 
(iii)  the random variables  $(\t_{i+1}-\t_i: i\in \bbZ \setminus \{0\})$ have the same law of $\tilde \t_1$. Such a renewal point process exists since  $\E[\tilde \t_1]<\infty$ (see \cite{DV} and \cite[Appendix C]{FMRT}).
 Due to the above properties (ii) and (iii) it is trivial to build  the renewal point process $\tau $ once the joint law of $(\t_0, \t_1)$ is determined. This last joint law can be recovered from the basic identity given by \cite[Eq.~(C.3)]{FMRT}. Since in what follows we only need  the law of $\t_1$ we explain how to get it. Taking \cite[Eq.~(C.3)]{FMRT} with test function $f\equiv 1$, one gets that $\bbP( 0\in \t)= 1/ \bbE[\tilde \t_1]$.  Let us now take a positive integer $m$ and, given   a  subset $\xi \subset \bbZ$ , we define   $f(\xi)=1$ if $m$ is the smallest element in $\xi\cap [1, +\infty)$, otherwise $f(\xi)=0$. Then applying  \cite[Eq.~(C.3)]{FMRT} to such a test function $f$ and using  that $\bbP( 0\in \t)= 1/ \bbE[\tilde \t_1]$,  we get
\begin{equation}
\bbP( \t_1 = m ) = \bbP( 0\in \t) \bbP(\tilde\t_1 \geq m ) = \frac{c}{\bbE[\tilde \t_1]}  m^{-\g} \,, \qquad m \in \N_+   \,.
\end{equation}
 From the above formula we have 
 \[ 
c_1 m^{1-\g}
 	\leq \bbP (\t_1 \geq m) 
 	\leq c_2 m^{1-\g}\,,\qquad  \ale{m \in \N_+}\,.
 \]

We recall our definition of the jump rates. For $k\in\Z$, we define
\begin{align*}
Z_k=
\begin{cases}
1\quad\mbox{if }k\notin \tau\,,\\
\frac12\quad\mbox{if }k\in \tau\,,
\end{cases}
\end{align*}
and set ${r(-k,-k+1)}={1}/{Z_k}$, $\o_k^+:=\frac{r(k,k+1)}{r(k,k-1)+ r(k,k+1)}$ and  $\o_k^-:=\frac{r(k,k-1)}{r(k,k-1)+ r(k,k+1)}$. Hence, 
\ale{$\rho_{-k}=\frac{r(-k,-k-1)}{r(-k,-k+1)}= A Z_k$ (recall that   $r(\cdot,\cdot-1)\equiv A$).} \ale{By construction, the model satisfies Assumption \ref{uffina}.}
 
Thanks to Lemma \ref{buddha} we can now calculate
\begin{align}\label{fragola}
\E[\bar S(\l)]
	=1+2\sum_{i=0}^\infty\E[ Z_0\cdots Z_i]\,(A\e^{-2\l})^{ i+1}\,.
\end{align}
We claim that 
\begin{align}\label{mora}
\E[ Z_0\cdots Z_n]=n^{1-\gamma+o(1)}\,.
\end{align}
In fact, on the one hand it holds
\begin{equation}\label{super}
\E[ Z_0\cdots Z_n]
	\ge  \E[ Z_1\cdots Z_n\mathds 1_{  \tau_1 > n    } ]/2 
	=\P(\tau_1>n )/2 
	\geq (c_1/2) (n+1)^{1-\g}   
	=n^{1-\gamma+o(1)}\,.
\end{equation} 
On the other hand, we take $A(n):=\lceil(\gamma-1)\log_2 n\rceil$  and we write
\begin{align}\label{spuma}
\E[ Z_0\cdots Z_n]
	=\E[ Z_0\cdots Z_n\mathds 1_{\{\tau_{A(n)}\leq n\}}]
		+ \E[ Z_0\cdots Z_n\mathds 1_{\{\tau_{A(n)}>n\}}]\,.
\end{align}
Notice that under the event $\{\tau_{A(n)}\leq n\}$ we have that 
\begin{equation}\label{spuma1}
 Z_0\cdots Z_n\leq 2^{- A(n)} \ale{\leq} n^{1-\gamma}\,.
\end{equation} The event $\{\tau_{A(n)}> n\}$, instead, implies that  the sum of the $A(n)$ addenda 
$\t_1$, $\t_2-\t_1$, $\t_3 -\t_2$,..., $\t_{A(n)}- \t_{A(n)-1}$ is larger than $n$. In particular, at least one of the above addenda is larger than $n /A(n)$. Since in any case $Z_0\cdots Z_n \leq 1$, we can bound
\begin{align}\label{spuma2}
\E[ Z_0\cdots Z_n\mathds 1_{\{\tau_{A(n)}>n\}}] 
	& \leq \P( \tau_{A(n)}>n) 
	\leq\P(\t_1>n/A(n))+\sum_{j=2}^{A(n)}\P( \t_j-\t_{j-1} > n/A(n))\nonumber\\
	& \leq c_2 (n/A(n)) ^{1-\g}+ A(n)\,c (n/A(n) ) ^{-\g}
	= n^{1-\g+o(1)}\,.
\end{align}
As a byproduct of \eqref{spuma}, \eqref{spuma1} and \eqref{spuma2} we conclude that 
$\E[ Z_0\cdots Z_n]
 \leq  
	n^{1-\gamma+o(1)}$.
The above result and \eqref{super} imply our claim \eqref{mora}.

\smallskip

Now we take $\g>2$, so that the series $\sum _{n=1}^\infty n ^{1-\g+o(1) }$ is summable. By \eqref{fragola} and \eqref{mora} we see that $\E[\bar S(\l)]< +\infty$ for all $\l\geq \l_+=\frac 12 \log A$, while $\E[\bar S(\l)]=+\infty$ for $\l<\l_+$. It follows that $v_X(\l)> 0$ for $\l\geq \l_+$, while $v_X(\l)\leq 0$  for $\l<\l_+$, that is, $v_X$ has a discontinuity in $\l_+$. We also observe that $\l_-$ is finite and $\l_-<\l_+$. In fact, \ale{by Proposition \ref{ferb} and Lemma \ref{buddha},}
 we have that $v_X(\l)<0$ iff $\E[\bar F(\l)]=1+2\sum_{i=0}^\infty \E[\rho_0^{-1}\cdots\rho_i^{-1}]\e^{2\l(i+1)}$ is finite. By \ale{Jensen's inequality} we can bound 
\begin{align*}
\E[\rho_0^{-1}\cdots\rho_i^{-1}]
	=A^{-i-1}\E[Z_0^{-1}\cdots Z_{-i}^{-1}]
	=A^{-i-1}\E[2^{N(i)}]
	\geq A^{-i-1}2^{\E[N(i)]}
	\geq A^{-i-1}2^{ci}
\end{align*}
where $N(i):=\#\{\tau \cap \{0,1,\dots,i\}\}$ and $c>0$ is a constant that only depends on $\gamma$ (here we are using the fact that $N(i)/i\to 1/\E[\tilde \tau_1]$ almost surely). Therefore $\E[\bar F(\l)]=+\infty$ if $2^c\e^{2\l}/A\geq 1$ and in particular $\l_-\leq \l_+-\tfrac c2\log 2$.


%
%
%
%

\section{Discussion of Example \ref{lorenzo4}} \label{lorenzo4_esteso}

It is not hard to show that this environment satisfies \ale{Assumptions \ref{uffina} and  \ref{jesus} (for the latter, observe that } the reflection of the Heads case has the same distribution of the Tails case, and vice versa). 

\smallskip

Let us first suppose that we have  heads.
Then  for each $m \in \bbZ$ we get 
\[ \rho_{2m+1} \rho _{2m+2}=\frac{ a_m^-}{a_m^+} \frac{a_m^+}{a_m^-}=1\,.
\]
As a consequence, we have for $i \geq 1$
\[
\rho_1 \rho_2 \cdots \rho _i = 
\begin{cases}
 1 & \text{ if $i$ is even} \,,\\
 \rho _i = \frac{ r_i^-}{r_i^+}=\frac{a_j^-}{a_j^+} & \text{ if $i=2j+1$}\,.
\end{cases}
\]
Hence for Heads we would have
\begin{equation}\label{passo1}
\sum_{i=1}^\infty\rho_1\rho_2\cdots\rho_i\e^{-2\l i}=
 \sum_{j=1}^\infty \e^{- 4 \l j}+
\sum_{j=0}^\infty\frac{a_j^-}{a_j^+}  \e ^{- 2\l (2j+1)}
\,.
\end{equation}

Let us now  suppose that we have tails.
Then  for each $m \in \bbZ$ we get  
\[ \rho_{2m} \rho _{2m+1}=\frac{ a_m^-}{a_m^+} \frac{a_m^+}{a_m^-}=1\,.
\]
As a consequence, we have for $i \geq 1$
\[
\rho_1 \rho_2 \cdots \rho _i = 
\begin{cases}
\rho_1 \rho_i=\frac{r_1^-}{ r_1^+} \frac{ r_i^-}{r_i^+}=  \frac{a_0^+  }{ a_0^-  }\frac{a_j^-}{a_j^+}& \text{ if $i=2j$}\,,\\
\ale{ \rho_1 = \frac{a_0^+  }{ a_0^-  } }& \text{ if $i$ is  odd} \,.
\end{cases}
\]
Hence for tails  we would have
\begin{equation}\label{passo2}
\sum_{i=1}^\infty\rho_1\rho_2\cdots\rho_i\e^{-2\l i}= \sum_{j=1}^\infty  \frac{a_0^+  }{ a_0^-  }\frac{a_j^-}{a_j^+} \e^{- 4 \l j}+
\sum_{j=0}^\infty   \frac{a_0^+  }{ a_0^-  }\   \e ^{- 2\l (2j+1)}
\,.
\end{equation}
Since $r_0^+=a_{-1}^-$ for Heads and $r_0^+= a_0^+$ for Tails,  from \eqref{passo1} and \eqref{passo2} we get
\begin{equation}
\begin{split}
&\E[\hat S(\l)]=\E\Big[\frac{\e^{-\l}}{r_0^+}\big(1+\sum_{i=1}^\infty\rho_1\rho_2\cdots\rho_i\e^{-2\l i}\big)\Big]\\
	&= \frac{\e^{-\l}}{2} \bbE[1/ a^-_{-1}] \big( 1+  \sum_{j=1}^\infty \e^{- 4 \l j}+
\sum_{j=0}^\infty \bbE\bigl[a_j^-/a_j^+\bigr]  \e ^{- 2\l (2j+1)} \big)\\
	&\qquad+ \frac{\e^{-\l}}{2} \big( \ale{1/\bbE[a_0^+]}+\sum_{j=1}^\infty  \bbE \bigl[ 1/ a_0^-  \bigr] \bbE [ a_j^-/a_j^+  \bigr] \e^{- 4 \l j}+
\sum_{j=0}^\infty \bbE \bigl[ 1/ a_0^- \bigr]   \e ^{- 2\l (2j+1)} \bigr)\\
	&=  \frac{\e^{-\l}}{2} B\big( 1+  \sum_{j=1}^\infty \e^{- 4 \l j}+
AB \sum_{j=0}^\infty   \e ^{- 2\l (2j+1)} \big) + \frac{\e^{-\l}}{2} B\big( 1+AB  \sum_{j=1}^\infty     \e^{- 4 \l j}+
 \sum_{j=0}^\infty    \e ^{- 2\l (2j+1)} \bigr).\end{split}
\end{equation}
Hence $\l_+=0$ (which by reflection invariance implies that also $\l_-=0$) and, for $\l>0$, 
\begin{equation}
\E[\hat S(\l)]
	=   \frac{B\e^{-\l}}{2(1-\e^{-4\l})}\big(2+ \e^{-2\l}(AB+1)+\e^{-4\l}(AB-1)\big)\,.
\end{equation}
From \eqref{gioventu} we finally get
\begin{equation}\label{pane}
v(\l)  
	= \frac{  2(1-\e^{-4\l})}{ B\e^{-\l} \big(2+ \e^{-2\l}(AB+1)+\e^{-4\l}(AB-1)\big)}\,,\qquad \l\geq0\,.
\end{equation}
By Taylor expansion of \eqref{pane},   since $2+ \e^{-2\l}(AB+1)+\e^{-4\l}(AB-1) =2(1 + AB- 3 \l AB +\l)\ale{+o(\l)} $, and since $(1+x)^{-1} = 1-x +\ale{ o(x)}$ we have
\begin{equation}
\begin{split}
 v(\l)& = \frac{4\l- 8 \l^2+o(\l^2) }{B (1-\l +o(\l))(1+AB) ( 1+\frac{1-3AB}{1+AB} \l +o(\l))}\\
 &=\frac{(4\l- 8 \l^2+o(\l^2)) (1+\l +o(\l))  ( 1-\frac{1-3AB}{1+AB} \l +o(\l)) }{B(1+AB)} \,.
\end{split}
\end{equation}
The above equation implies that
\begin{equation}
\begin{split}
v(\l)& = \frac{ 4}{ B(1+AB)} \l+ \frac{  -8+4-4\frac{1-3AB}{1+AB} }{ B(1+AB)} \l^2 +o(\l^2)\\
&= \frac{ 4}{ B(1+AB)} \l +8 \frac{AB-1}{ B(1+AB)^2} \l^2+o(\l^2)\,.
\end{split}
\end{equation}
We then conclude that $ \partial^+ _{\l,\l}v(0)=16\frac{AB-1}{ B(1+AB)^2} $.

\smallskip

From this expression,  the second right derivative in $\l=0$ is null if and only if $AB=1$, that is, when $a^+_0$ is almost surely a constant. In all the other cases, the right \ale{second  derivative} in $\l=0$ is different from $0$. Since this model satisfies Assumption 1,  \ale{ this conclusion  is absurd 
by Proposition \ref{redenzione}.}
\bigskip

\appendix
\section{A result on analytic functions}
\begin{Lemma}\label{latticini}
Let $a,b,c $ be positive numbers with $a<b$ and let $(a_n )_{n\geq 0}$ be a sequence such that $\sum_{n=0}^\infty |a_n| \e^{-c  \l n}<\infty$ for any $\l \in (a,b)$. Then the function $f(\l):= \sum_{n=0}^\infty a_n \e^{-c  \l n}$ is well defined and analytic  for $\l\in (a,b)$.
\end{Lemma}
\begin{proof}
The function $f(\l)$ is well defined since  by hypothesis  the series is absolutely convergent. To prove analiticity we apply
the Theorem of  Pringsheim-Boas (cf. \cite[Thm.~3.1.1]{KP}). To this aim we first need to show that $f$ is $C^\infty$ on $(a,b)$. This follows easily from the dominated convergence theorem, which also implies that 
$f^{(k)}$, the $k^{th}$--derivative of $f$,  has the form
\[  f^{(k)} (\l):= \sum_{n=0}^\infty a_n \e^{-c  \l n} (-c n)^k\,,
\]
where the series in the r.h.s. is absolutely convergent.

Defining $g(\l):= \limsup_{k\to \infty} |f^{(k)}(\l)/k!|^{1/k}$, we need to show that for any $\l_0\in (a,b)$ $g(\l)$ is bounded from above uniformly  as $\l$ varies in a neighbourhood of $\l_0$. Then the Theorem of  Pringsheim-Boas  would imply the analiticity of $f$. To upper bound $g$  we use that
$
\frac{(c n)^k}{k!}\leq \frac{e^{c\, n \ep  }}{\ep^k}$ for $\ep>0$ 
to estimate
\ale{\begin{equation}\label{piccola_anima}
\Big | \frac{f^{(k)} (\l)}{k!}\Big |\leq  \sum_{n=0}^\infty \frac{|a_n|}{k!} \e^{-c  \l n} (c n)^k\leq \frac{1}{\ep^k}
\sum_{n=0}^\infty |a_n| \e^{-c ( \l-\ep) n}=:\frac{ h(\l-\ep)}{\ep^k}\,.
\end{equation}}
By hypothesis $h(\l-\ep)$ is finite if $\l-\ep \in (a,b)$. 
We now take  $\l$ such that 
  $|\l -\l_0| \leq \ep/10$ where  $\ep$ is defined  as half of the distance   between $\l_0$ and $\{a,b\}$. Then both $\l$ and $\l-\ep$ are in $(a,b)$ and by  \eqref{piccola_anima} we get that $g(\l) \leq 1/ \ep$. 
\end{proof}

%
%
%
\section{Proof of Proposition \ref{ariccia}}\label{nonlaso}
We follow the proof of \cite[Theorem 2.1.9]{Z} and adapt the \ale{arguments} to the continuous time case. \ale{  As already mentioned, the main difference lies  in   the proof of the result analogous to \cite[Lemma 2.1.17]{Z}, where new  phenomena  have to be controlled.}


We denote by $\theta$ the shift on the space $\Theta$  of environments. In particular, we have $(\theta \xi)_x := \xi_{x+1}=(r_{x+1}^-,r_{x+1}^+)$. 
For $n\in\Z$, we introduce the hitting times 
\begin{align}\label{vincent}
T_n= \ale{T_n(\xi, \l)}:=\inf\{t\geq 0\,:\,\ale{Y_t^{\xi,\l}}=n\}\,,
\end{align}
with the convention that the infimum of an empty set is $+\infty$. We also set $\tau_0=0$ and
\begin{align*}
\tau_n:=\ale{\tau_n(\xi,\l)}=T_n-T_{n-1}\qquad \mbox{for $n\geq 1$}\,,\\
\tau_{-n}=T_{-n}-T_{-n+1}\qquad \mbox{for $n\geq 1$}\,.
\end{align*}
As in \cite[Lemma 2.1.10]{Z}, one can prove that if \ale{$\limsup _{t\to\infty}Y_t^{\xi,\l}=+\infty$} almost surely, then the sequence $\{\tau_i\}_{i\geq 1}$ is stationary and ergodic.
The idea is then to apply the ergodic theorem to the sequence \ale{$\{\tau_i\}_{i\geq 1}$}. We prove the equivalent of \cite[Lemma 2.1.12]{Z}:
\begin{Lemma}\label{ciriola}
We have that
\begin{align*}
(a)\qquad &\E[E^{\xi,\lambda}_0[\tau_1]]=\E[\hat S(\l)]\,,\\
(b)\qquad &\E[E^{\xi,\lambda}_0[\tau_{-1}]]=\E[\hat F(\l)]\,.
\end{align*}
\end{Lemma}
\begin{proof}
We just show $(a)$, since $(b)$ can be proved in an identical way. For each environment $\xi$ we have
\begin{align}\label{bandito1}
E^{\xi,\lambda}_0[\tau_1]
	=\frac{1}{r_0^+(\l)+r_0^-(\l)}+\ale{\frac{r_0^-(\l)}{r_0^+(\l)+r_0^-(\l)} } \big(E^{\theta^{-1}\xi}_0[\tau_1]+E^\xi_0[\tau_1]\big)\,.
\end{align}
Manipulating the last expression we obtain
\begin{align}\label{bandito2}
E^{\xi,\lambda}_0[\tau_1]
	=\frac{1}{r_0^+(\l)}+\rho_0(\l) E^{\theta^{-1}\xi}_0[\tau_1]\,,
\end{align}
so that by iteration we \ale{get, for any  integer $m\geq 0$,}
\begin{align}\label{bandito3}
E^{\xi,\lambda}_0[\tau_1]
	=&\frac{1}{r_0^+(\l)}+\frac{1}{r_{-1}^+(\l)}\rho_0(\l)+\frac{1}{r_{-2}^+(\l)}\rho_0(\l)\rho_{-1}(\l)+\dots + \nonumber\\
	&+\frac{1}{r_{-m}^+(\l)}\rho_0(\l)\rho_{-1}(\l)\dots\rho_{-m+1}(\l)+\rho_0(\l)\rho_{-1}(\l)\dots\rho_{-m}(\l) E_0^{\ale{\theta^{-m-1}\xi}}[\tau_1]\,.
\end{align}
By positivity of all the summands, we deduce, for all environments $\xi$ and for all $m\in\N$, that $E^{\xi,\lambda}_0[\tau_1]\geq\frac{1}{r_0^+(\l)}+ \dots +\frac{1}{r_{-m}^+(\l)}\rho_0(\l)\rho_{-1}(\l)\dots\rho_{-m+1}(\l)$. Letting $m\to\infty$ we obtain that $E^{\xi,\lambda}_0[\tau_1]\geq \hat S(\l)$ and in particular $\E[E^{\xi,\lambda}_0[\tau_1]]\geq \E[\hat S(\l)]$. This already \ale{shows that $(a)$ holds   if $\E[\hat S(\l)]=+\infty$}.

\smallskip

Conversely, let us suppose that $\E[\hat S(\l)]$ is finite. We claim that $E^{\xi,\lambda}_0[\tau_1]\leq \hat S(\l)$, which would complete the proof of the lemma thanks to the complementary bound proved above. Repeating \eqref{bandito1} on the event $\{\tau_1\leq M\}$ we obtain the equivalent of \eqref{bandito2}, which now reads
\begin{align*}
E^{\xi,\lambda}_0[\tau_1\mathds{1}_{\{\tau_1\leq M\}}]
	\leq\frac{1}{r_0^+(\l)}+\rho_0(\l) E^{\theta^{-1}\xi}_0[\tau_1\mathds{1}_{\{\tau_1\leq M\}}]\,.
\end{align*}
As in \eqref{bandito3} we iterate this relation and bound
\begin{align}\label{bandito3bis}
E^{\xi,\lambda}_0[\tau_1\mathds{1}_{\{\tau_1\leq M\}}]
	\leq\hat S(\l)+\rho_0(\l)\rho_{-1}(\l)\dots\rho_{-m}(\l) M\,, \qquad \ale{\forall m \in \N}\,.
\end{align}
Since we assumed that 
$\E[\hat S(\l)]=\E[\frac{1}{r_0^+(\l)}+\frac{1}{r_{-1}^+(\l)}\rho_0(\l)+\frac{1}{r_{-2}^+(\l)}\rho_0(\l)\rho_{-1}(\l)+\dots]<\infty$, 
we must have that, $\P$-a.s., 
$\lim_{m\to\infty} \frac{1}{r_{-m}^+(\l)}\rho_0(\l)\rho_{-1}(\l)\dots\rho_{-m+1}(\l)=0$. 
We fix now an $\varepsilon>0$ small enough. By ergodicity, for almost every $\xi$ there exists \ale{ a $\xi$--dependent  infinite sequence $m_1<m_2<m_3<\dots$} such that $r_{-m_k}^+(\l)\geq \varepsilon$. But this, together with the previous observation, implies that also 
$\lim_{\ale{k\to\infty}}\rho_0(\l)\rho_{-1}(\l)\dots\rho_{-m_k+1}(\l)=0$ 
for $\P$-a.e.~$\xi$.  In particular, rewriting \eqref{bandito3bis} as $E^{\xi,\lambda}_0[\tau_1\mathds{1}_{\{\tau_1\leq M\}}] \leq \hat S(\l)+\ale{\rho_0(\l)\rho_{-1}(\l)\dots\rho_{-m_k}(\l)} M$ and sending $k\to\infty$, we have $E^{\xi,\lambda}_0[\tau_1\mathds{1}_{\{\tau_1\leq M\}}]\leq \hat S(\l)$ for almost every $\xi$. Finally, we let $M\to\infty$ and use dominated convergence to obtain the claim.
\end{proof}

\medskip

We are finally ready to conclude the proof of Proposition \ref{ariccia}.  We prove the continuous version of \cite[Lemma 2.1.17]{Z}, which will require some work.
We can build the random walk \mic{$Y_t^{\xi,\l}$} as follows. The environment $\{ r_x^\pm(\l) \}_{x\in \bbZ}$ is defined as usual on a probability space with law $\bbP$.  
We introduce  a sequence $\g=(\g_n )_{n\geq 0}$ of i.i.d.~uniform random variables with value on $[0,1]$ and another sequence $W:=(W_n )_{n\geq 0}$ of i.i.d.~exponential  random variables with mean $1$, both defined on some other probability space with probability $Q$ and independent of each other.
On the product space with probability $\bbP\otimes  Q $ we define the following objects (in what follows, when we write ``a.s.", we mean ``$\bbP\otimes Q$ --a.s.'').
 We iteratively define
$X_0:=0$ and 
\[ X_{n+1} :=X_n+ \mathds{1}( \g_n \leq \o_{X_n}^+(\l)) -  \mathds{1}( \g_n > \o_{X_n}^+(\l)) \,.
\]
We also define  $\hat W_n:= (r_{X_n}^+(\l)+r_{X_n}^-(\l))^{-1}  W_n$. Note that $\hat W_n$ is an exponential variable of parameter  $r_{X_n}^+(\l)+r_{X_n}^-(\l)$. We set $U_0:=0$ and $U_n:= \hat W_0+\hat W_1+\cdots +\hat W_{n-1}$ for $n\geq 1$. Take $c_0>0$ such that $r_0^+(\l)+r_0^-(\l) \leq c_0$ with positive probability. 
By \cite[Thm.~2.1.2]{Z} the random walk $(X_n)_{n\geq 0}$ explores a.s.~at least one half-line of $\bbZ$. In particular, it will visit infinitely many points $x$ such that $r_x^+(\l)+r_x^-(\l) \leq c_0$.
It then follows that 
$\lim _{n \to \infty} U_n=+\infty$ a.s. Therefore we can define a.s.~for all $t\geq 0$ the state  \mic{$Y_t^{\xi,\l}$ as  $Y_t^{\xi,\l}:= X_{n(t)}$} where $n(t)$ is the unique  integer such that $U_{n(t)} \leq t< U_{n(t)+1}$.

\begin{Claim}\label{pisa10}
It holds
$$
\limsup _{t\to \infty} \frac{\mic{Y_t^{\xi,\l}}}{t}\leq \frac{1}{\bbE[ \hat S(\l)]} \qquad a.s. 
$$\end{Claim}
  \begin{proof}[Proof of Claim \ref{pisa10}] Since \mic{$Y_t^{\xi,\l}$} is a random--time change of  the associated jump process $X_n$, 
by \cite[Thm.~2.1.2]{Z} we have that either $\lim_{t\to \infty} \mic{Y_t^{\xi,\l}}=-\infty$  a.s., or 
$\limsup_{t\to \infty} \mic{Y_t^{\xi,\l}} =+\infty$ a.s. In the first case we have nothing to prove since $\lim_{t\to \infty} {\mic{Y_t^{\xi,\l}}}/{t} \leq 0$  a.s. Hence we can suppose that $\limsup_{t\to \infty} \mic{Y_t^{\xi,\l}} =+\infty$.

Given $t$ we call $k(t)$ the unique integer such that $T_{k(t)}\leq t <T_{k(t)+1}$ (recall definition \eqref{vincent}). Equivalently, $k(t) = \sup \{ Y^{\xi, \l}_s \,:\, 0\leq s \leq t\}$. Since  $\limsup_{t\to \infty} \mic{Y_t^{\xi,\l}} =+\infty$ a.s., we have that $\lim _{t\to \infty} k(t)= +\infty$  a.s.  As in \cite{Z}, we combine Lemma \ref{ciriola} with the ergodicity of the sequence $\{\tau_i\}_{i\geq 1}$ to obtain
\begin{align*}
\lim_{n\to\infty}\frac{T_n}{n}=\lim_{n\to\infty}\frac{\sum_{i=1}^n\tau_i}{n}\xrightarrow{n\to\infty}\E[\mic{E^{\xi,\l}_0}[\tau_1]]=\E[\hat S(\l)]=:\a\quad \text{a.s. }
\end{align*}
As a consequence, 
$\lim_{n\to\infty}{n}/{T_n}=1/\a$ a.s.
Since $\lim _{t\to \infty} k(t)= +\infty$  a.s.,  we have that $\lim_{t\to\infty}{k(t)}/{T_{k(t)}}=1/\a$ a.s. By definition of $k(t)$ we have $\mic{Y_t^{\xi,\l}}\leq k(t) $, so that
\[ \limsup_{t\to\infty} \frac{\mic{Y_t^{\xi,\l}}}{t} \leq  \limsup _{t\to \infty} \frac{k(t)}{T_{k(t)}}=\frac{1}{\a} \qquad \text{ a.s.}
\] 
This concludes the proof of our claim.
\end{proof}

Note that by similar arguments one can prove that  $\liminf _{t\to \infty} {\mic{Y_t^{\xi,\l}}}/{t}\geq -{1}/{\bbE[ \hat F(\l)]}$ a.s. 
In particular, if $\bbE[ \hat S(\l)]=\infty$ and $\bbE[ \hat F(\l)]=\infty$, then we have that 
$\lim_{t\to \infty}{\mic{Y_t^{\xi,\l}}}/{t}=0$ a.s. This concludes the proof of Proposition \ref{ariccia}--(c).

\begin{Claim}\label{pisa2} 
If  $\a:=\bbE[ \hat S(\l)]\in (0,\infty)$, then 
$$
\liminf_{t\to \infty} \frac{\mic{Y_t^{\xi,\l}}}{t} \geq \frac{1}{\a}\qquad a.s.
$$
\end{Claim} 
We point out that Claim \ref{pisa2} together with Claim \ref{pisa10} gives  Proposition \ref{ariccia}--(a). 
By similar arguments we can also get Proposition \ref{ariccia}--(b) and the proof of Proposition \ref{ariccia} is concluded.
\begin{proof}[Proof of Claim \ref{pisa2}] Recall the arguments and definitions in the proof of Claim \ref{pisa10}.
By Lemma \ref{ciriola}--(a) $\t_1$ is finite  a.s. By iteration one gets that $T_n=\t_1+\cdots +\t_n$ is finite a.s. (recall \eqref{vincent}). As a consequence $\limsup _{t\to +\infty} Y_t^{\xi,\l}=+\infty $ and $\lim _{t\to +\infty} k(t)=+\infty$ a.s. 
 Given $\ell\in \N$ call 
$m_\ell:=\inf \{Y^{\xi, \l}_s\,:\, T_\ell \leq s < T_{\ell+1} \}$.
Note that $m_\ell$ depends only on $\mic{\xi}$ and $\g$. We have
\[    m_{k(t) }\leq  \inf \{Y^{\xi, \l}_s\,:\, T_{k(t)} \leq s\leq t\} \leq \mic{Y_t^{\xi,\l}}\,.
\]
To conclude, since $k(t) \to \infty$ a.s., we would only  need to show that, fixed $\ep>0$,  it holds a.s.: 
$m_\ell  >  \ell-2 \ep \ell$ for each $\ell$ large enough.  Indeed, this fact would imply that, for any fixed $\ep>0$, a.s.~$\mic{Y_t^{\xi,\l}}\geq k(t) -2 \ep k(t)$ for $t $ large. Hence, 
\[ 
\liminf_{t\to\infty} \frac{\mic{Y_t^{\xi,\l}}}{t} 
	\geq \liminf_{t\to \infty} (1-2\ep) \frac{k(t)}{t}
	\geq  \liminf_{t\to \infty} (1-2\ep) \frac{k(t)}{T_{k(t)\mic{+1} } }
	=\frac{1-2\ep}{\a} \qquad \text{ a.s.}
\]
Thanks to the arbitrariness of $\ep$ we would  get the thesis.

It remains therefore to prove that 
\begin{equation}\label{miraggio}
\limsup_{\ell \to \infty} \mathds{1} ( m_\ell  \leq  \ell-2 \ep \ell)=0 \qquad \text{ a.s.}
\end{equation}
Take $M>0$ so large that $\bbP(r_{0}^+(\l)+r_{0}^-(\l) \leq  M)> 1-\ep/2$. By the ergodic theorem
\[ 
\lim_{\ell \to \infty} \frac{1}{\ell}\sum_{j=1}^{\mic{\ell}} \mathds{1}( r_{j}^+(\l)+r_{j}^-(\l) \leq M) =  \bbP( r_{0}^+(\l)+r_{0}^-(\l) \leq  M)> 1-\frac{\ep}{2} \qquad 
\text{$\bbP$--a.s. }
\]
As a consequence, for $\bbP$--a.e.~$\mic{\xi}$ (let us say for all $\mic{\xi} \in \cA$) there exists $\ell_0(\mic{\xi})$ such that
\[
\sum_{j=1}^{\mic{\ell}} \mathds{1}( r_{j}^+(\l)+r_{j}^-(\l) \leq M) \geq   \ell (1-\ep) \qquad \forall \ell \geq \ell_0(\mic{\xi})\,.
\] 
This implies that
$\sharp \{ j \in [\ell - 2\ep \ell, \ell]\cap \N\,:\,  r_{j}^+(\l)+r_{j}^-(\l) \leq M \} \geq \ep \ell$, for all $\ell\geq \ell_0(\mic{\xi}) $.
Suppose to know $\mic{\xi}$ (with $\mic{\xi}\in \cA$), $\g$, that $m_\ell \leq \ell -2\ep \ell$ and that  $\ell\geq \ell_0(\mic{\xi})$. Then the time that the continuous time random walk needs to reach
$\ell - 2\ep \ell $  after visiting for the first time $\ell$ is stochastically dominated from below by the sum  of $\ep \ell$ i.i.d.~exponential random variables with mean $1/M$.  Let $A_1,A_2, \dots$ be i.i.d.~exponential random variables with mean $1/M$.   
Then, 
fixed $\d>0$ such that $\d/\ep < 1/M$, by Cram\'er theorem we  have  
\[P\Big( \sum_{j=1}^{\ep \ell} A_j \leq \d \ell \Big) 
	= P\Big( \frac{1}{\ep \ell} \sum_{j=1}^{\ep \ell} A_j \leq {\d}/{\ep}  \Big)\leq \e^{ - c \ep \ell} \qquad \forall \ell \geq \ell_1
\]
for suitable constants  $c, \ell_1 >0$.
This bound combined with the stochastic domination implies that
\[ 
\mic{\P}\otimes Q ( \t_{\ell+1} \leq  \d \ell\,|\,  \mic{\xi} , \g ) 
	\leq  \e^{ - c \ep \ell}\,, 
\]
on the event  $ E_\ell (\mic{\xi},\g):=\{  \mic{\xi} \in \cA, \, m_\ell \leq \ell -2\ep\ell,\, \ell  \geq \ell_0 (\mic{\xi})\lor \ell_1 \}$. 
 Hence, 
\[ 
\P \otimes Q 
\bigl  (\{\t_{\ell+1} \leq  \d \ell\} \cap E_\ell (\mic{\xi},\g)\bigr)
 \leq \e^{ - c \ep \ell}\,.\]
By the Borel--Cantelli lemma we get that there exists a random integer $L$ such that, for $\ell \geq L$, the event  $\{\t_{\ell+1} \leq  \d \ell\}\cap E_\ell (\mic{\xi},\g)$ does not take place. With more elegance, we can write
\[ 
\limsup_{\ell \to \infty} \mathds{1}\bigl(\{ \t_{\ell+1} \leq  \d \ell\}\cap E_\ell (\mic{\xi},\g)\bigr)=0 \qquad \text{a.s.}
\]
which implies that 
\begin{equation}\label{soffro1}
 \limsup_{\ell \to \infty} \mathds{1}( \t_{\ell+1} \leq  \d \ell,\,   m_\ell \leq \ell -2\ep\ell )=0 \qquad \text{a.s.}
\end{equation}
Now observe that, by Lemma \ref{ciriola} and the discussion preceding it, $T_\ell /\ell \to 1/\a$ a.s., so that $T_{\ell+1} /\ell \to 1/\a$ a.s., too.
As a consequence, $\t_{\ell+1}/\ell= (T_{\ell+1} /\ell)-(T_\ell /\ell )\to 0$ a.s.
It then follows that 
\begin{equation}\label{soffro2}
\lim _{\ell \to \infty } \mathds{1} ( \t_{\ell+1} \geq \d \ell)=0 \qquad \text{ a.s. }
\end{equation}
Since 
\[ 
\mathds{1} ( m_\ell \leq \ell -2\ep\ell )
	\leq \mathds{1}( m_\ell \leq \ell -2\ep\ell ,\, \t_{\ell+1} \leq \d\ell)
		+  \mathds{1} ( \t_{\ell+1} \geq \d \ell)\,,
\]
as a byproduct of \eqref{soffro1} and \eqref{soffro2} we conclude that 
 $\limsup_{\ell \to \infty} \mathds{1}(  m_\ell \leq \ell -2\ep\ell )=0 $ a.s. This concludes the proof of \eqref{miraggio}.\end{proof}

\section{Proof of Theorem \ref{baldo}} \label{spagna}
The proof of the CLT is the same as in \cite{Z} with two exceptions: The different $\sigma$-algebra for condition \eqref{soreta} and the unique step in \cite{Z} where \eqref{erica2} is used.

We start from the first issue. Call $\mathcal F_k$ the $\sigma$-algebra generated by $(\theta^i \bar S(\l))_{i\leq k}$. In the proof of Theorem 2.2.1 in \cite{Z} one only needs inequality \eqref{soreta} with $\cG_{-n}$ replaced by $\cF_{-n}$. This is indeed automatically satisfied when \eqref{soreta} holds since $\cF_k\subset \cG_k$ and therefore, for each random variable $Z$, 
\begin{align*}
\E\Big[\E[Z\,|\,\cF_k]^2\Big]
	=\E\Big[\E\big[\E[Z\,|\,\cG_k]\,|\,\cF_k\big]^2\Big]
	\leq \E\Big[\E\big[\E[Z\,|\,\cG_k]^2\,|\,\cF_k\big]\Big]
	=\E\Big[\E[Z\,|\,\cG_k]^2\Big]\,,
\end{align*}
where we have used Schwarz inequality.

\smallskip

We move to the second issue. As in \cite{Z},  we set $f(x,n, \o):= x-v_X(\l)n + h(x, \o) $, $h(0,\o):=0$, $\D(x,\o):= h(x+1, \o)- h (x,\o)$, 
$\D(x,\o):= -1+v_X(\l) \bar S( \theta ^x \o)$, $\bar M_k:= f(\ale{X_k^{\o,\l}}, n, \o)$, $\bar M_0:=0$. Then, \eqref{erica2} is used in \cite{Z} to derive Eq. (2.2.8) there, and in particular that for 
$\bbP$--a.a. $\o$ the rescaled martingale $\bar M_n/\sqrt{n}$
  weakly converges to 
	$ \mathcal N(0,\sigma^2_1(\l))
$ under \ale{$\bbP_0^{\o,\l}$}.  Hence, we need to show that \eqref{erica} suffices to this task. 
To this aim,
 in what follows we write \ale{$P^\l_\o$} for the law on the path space $\O^{\N}$ of  the environment viewed 
 from the walker when the latter starts at the origin in the environment $\o$. \ale{$E_\o^\l$ will denote the associated expectation}.  We set $\bar{\o}:= (\bar{\o}_k)_{k \in \N}$ and $\bar{\o}_k:= \t_{\ale{X^{\o,\l}_k}} \o$.  By working with the law \ale{$P_\o^\l$} we think of \ale{$X_k^{\o,\l}$} as an additive functional of $\bar{\o}$.
 Since moreover $\bbQ_\l$ is mutually absolutely continuous w.r.t. $\bbP$, 
 we only need to prove that for $\bbQ_\l$--a.a. $\o$   \ale{the }
  martingale $\bar M_n/\sqrt{n}$
  weakly converges to 
	$ \mathcal N(0,\sigma^2_1(\l))
$ under \ale{$P_\o^\l$}.  
  This is indeed the same approach used in \cite{Z}, restated with our notation. As there, we apply \cite[Lemma 2.2.4]{Z} with $Z_k$ defined as the martingale difference $Z_k:= \bar{M}_k -\bar{M}_{k-1}$ and $\cF_k$ given by the $\s$--algebra generated by
$\bar{\o}_0$, $\bar{\o}_1$,...,$\bar{\o}_k$.
By 
straightforward  computations 
we get
\begin{equation}\label{samarcanda}
\begin{split}
Z_{k+1}& =v_X(\l)  \mathds{1}\bigl( \ale{X^{\o,\l}_{k+1}=X_k^{\o,\l}}+1\bigr) \bigl(  \bar S (\bar{\o}_k)-1  \bigr)\\
&- 
v_X(\l)  \mathds{1}\bigl( \ale{X^{\o,\l}_{k+1} = X_k^{\o, \l}}-1\bigr) \bigl(  \theta^{-1}\bar S (\bar{\o}_k)+1  \bigr)\,.
\end{split}
\end{equation}

The verification, for $\bbQ_\l$--a.a. $\o$,   of Condition (a) in \cite[Lemma 2.2.4]{Z}  is as in \cite{Z}. The core is to check, for $\bbQ_\l$--a.a. $\o$,  Condition (b) in \cite[Lemma 2.2.4]{Z} using only \eqref{erica} instead of \eqref{erica2}.  We recall that in our context  Condition (b)  states that  $ \frac{1}{n}\sum _{k=1}^{n} \ale{E_\o^\l}\bigl[ Z_{k+1}^2 \mathds{1} (\ale{ |Z_{k+1}|}>\d \sqrt{n} ) \bigr]$ converges to $0$ as $n\to \infty$, given $\d>0$. By Markov's inequality, we only need to show that 
\begin{equation}\label{sirenetta}
 \lim _{n\to \infty} n^{-\frac{\ep }{2}} \frac{1}{n}\sum _{k=1}^{n} \ale{E_\o^\l}\bigl[ \ale{|Z_{k+1}|}^{2+\ep} \bigr] =0 \,.
\end{equation}
Due to \eqref{samarcanda}, $\ale{E_\o^\l}\bigl[ Z_{k+1}^{2+\ep} \bigr] = v_X(\l)^{2+\ep} \ale{E^\l_\o}[ f_1 (\bar{\o}_k)+ f_2(\bar{\o}_k)]$, where $f_1(\o):=  \bigl(  \bar S (\o )-1  \bigr)^{2+\ep}\o_0^+$ and $f_2(\o):=  \bigl(  \theta^{-1}\bar S (\o)+1  \bigr)^{2+\ep} \o_0^{-}$. Due to \eqref{erica} the nonnegative functions $f_1,f_2$ are in $L^1(\bbQ_\l)$. Hence, we  get \eqref{sirenetta} for $\bbQ_\l$--a.a. $\o$ by Lemma \ref{befana100} below.

\begin{Lemma}\label{befana100}
Given a nonnegative function $f\in L^1(\bbQ_\l)$, for \ale{$\bbQ_\l$}--a.a. $\o$ it holds
\begin{equation}
 \lim _{n\to \infty} n^{-\frac{\ep }{2}} \frac{1}{n}\sum _{k=1}^{n} \ale{E^\l_\o}\bigl[f(\bar{\o}_k)  \bigr] =0 \,.
\end{equation}
\end{Lemma}
\begin{proof}
We fix positive numbers $\a,\g$ such that $\g< \ep/2$ and $\a \g >1$. We define $A_n$ as
\[ A_n:=\Big \{\o \in \O\,:\,  \ale{E^\l_\o}\bigl[ F_n \bigr]> n^\g\Big \}\,,\qquad F_n( \bar \o):= n^{-1} \sum_{k=1}^n f( \bar \o_k)\,.
\]
By Markov's inequality we have 
\begin{equation}\label{agliari}
 \bbQ_\l(A_n) \leq n^{-\g} \bbQ_\l\Big[   \ale{E^\l_\o}\bigl[ F_n \bigr]\Big]= n^{-\g} \| F_n \|_{L^1( \bbQ_\l \otimes\ale{ P^\l_\o})}  \,.
\end{equation}
Due to \cite[Corollary 2.1.25]{Z} the measure $ \bbQ_\l \otimes \ale{P^\l_\o}$ on $\O^{\N}$ is stationary and ergodic (w.r.t. time--shifts), hence by the $L^1$--Birkhoff ergodic theorem $F_n$ converges to $ \bbQ_\l[ \ale{ E^\l_\o}(f)]$ in $L^1( \bbQ_\l \otimes \ale{P^\l_\o})$ (here we use that $f\in L^1(\bbQ_\l)$). This automatically implies the convergence of the $L^1$--norms. As a byproduct with  \eqref{agliari} we get that 
$ \bbQ_\l(A_n) \leq C n^{-\g}
$ for some $n$--independent positive constant $C$. Setting now $n_j:= j^\a$, since $\a\g>1$ and therefore  $\sum_{j=1}^\infty n_j ^{-\g}<\infty$, by Borel--Cantelli lemma we conclude that for $\bbQ_\l$--a.a. $\o$ it holds $\o \not \in A_{n_j}$ for $j \geq j_0(\o)$. Hence, 
for $\bbQ_\l$--a.a. $\o$,  it holds $\ale{E^\l_\o}[ F_{n_j}] \leq n_j ^{\g}$ for $j \geq j_0(\o)$. Take such an environment $\o$ and take $n \geq n_{j_0(\o)}$. Then there exists $j\geq j_0(\o)$ such that $j^\a=n_j \leq n < n_{j+1}=(j+1)^\a$. Using that $f\geq 0$, for some $n$--independent constant $C'>0$ we can bound 
\begin{equation}
\ale{E^\l_\o}\bigl[ F_n\bigr] \leq \frac{n_{j+1}}{n_j} \ale{E_\o^\l} \bigl[ F_{n_{j+1} }\bigr]\leq  \frac{n_{j+1}}{n_j}  n_{j+1}^{\g} \leq n^{\g}  \frac{n_{j+1}^{1+\g} }{n_j^{1+\g}} \leq C' n^\g\,.
\end{equation}
Since $ \g<\ep/2$ and  $n^{-\frac{\ep }{2}} \frac{1}{n}\sum _{k=1}^{n} \ale{E_\o^\l}\bigl[f(\bar{\o}_k)  \bigr] = n^{-\frac{\ep }{2}}  F_n (\bar \o)$, we get the thesis.
\end{proof}


\section{Proof of Theorem \ref{suppongo}}\label{larome}

  \ale{ Assumption \ref{uffina} is satisfied due  to Remark \ref{fieno}.
As already observed, since the conductances are i.i.d., also Assumption \ref{jesus}  is satisfied, i.e. the environment is invariant under reflection.  As a consequence, the law of $(X^{\o, \l}_n)_{n\geq 0}$ under $\bbP \otimes P_0^{\o,\l}$ equals the law of $(-X^{\o, -\l}_n)_{n\geq 0}$ under $\bbP \otimes P_0^{\o,-\l}$. In particular, it holds $v_X(-\l)= - v_X(\l)$  and,  if the annealed CLT \eqref{diluvio} holds for  $\l >0$, then the same formula \eqref{diluvio} holds by replacing $\l $ with $-\l$ and taking $\s^2(-\l):= \s^2(\l)$.  It remains therefore to prove the annealed CLT and identity  \eqref{rey} for $\l>0 $. From now on we restrict to $\l>0$.
 To get the annealed CLT, due to Theorem \ref{baldo} and since $\bbE[\bar{S}(\l)]<\infty$, we only need to verify  Assumption \ref{gatto}   with $\cG_k$ being the $\sigma$-algebra generated by $(\rho_i:\,i\leq k)$. 
To this aim we first observe that, since 
 $\rho_i=\frac{c_{i-1}}{c_i}$,} we have
\begin{align*}
U=\sum_{i=0}^\infty \frac{c_{-i-1}}{c_0}\e^{-2\l(i+1)} 
	\qquad\mbox{and}\qquad
V=\sum_{i=1}^\infty \frac{c_{0}}{c_i}\e^{-2\l i} \,.
\end{align*}
Let $i\geq 0$ and $j\geq 1$. Since $\rho_0\cdots \rho_{-i}=c_{-i-1}/c_0$ and $\rho_1\cdots\rho_j=c_0/c_j$, we have the identities $\E[\rho_0^{2+\ep}\cdots \rho_{-i}^{2+\ep}]=\E[c_{0}^{2+\ep}]\E[1/c_0^{2+\ep}]$ and
$\E[\rho_0^{2+\ep}\cdots \rho_{-i}^{2+\ep}\rho_1\cdots\rho_j]=\E[c_{0}^{2+\ep}]\E[1/c_0^{1+\ep}]\E[1/c_0]$ and $\E[\rho_0^{2}\cdots \rho_{-i}^{2}]=\E[c_{0}^{2}]\E[1/c_0^{2}]$. \ale{Trivially, $\bbE[U^2]<\infty$. By  Propositions \ref{alternativa1} and \ref{alternativa2}, Assumption \ref{gatto} is therefore verified.}

\smallskip

To compute $\sigma^2(\l)$ we observe that
\begin{align*}
\E[U]=\E[V]
	&=AB\frac{\e^{-2\l}}{1-\e^{-2\l}}\\
\E[U^2]&=D\E\Big[\sum_{i=0}^\infty c_{-i-1}^2\e^{-4\l(i+1)}+\sum_{i\not = j=0}^\infty c_{-i-1}c_{-j-1} \e^{-2\l(i+1)}\e^{-2\l(j+1)}\Big]\\
	&=CD \frac{\e^{-4\l}}{1-\e^{-4\l}}+A^2D\Big(\frac{\e^{-4\l}}{(1-\e^{-2\l})^2}-\frac{\e^{-4\l}}{(1-\e^{-4\l})}\Big)\\
	&=CD \frac{\e^{-4\l}}{1-\e^{-4\l}}+A^2D\frac{2\e^{-6\l}}{(1-\e^{-2\l})(1-\e^{-4\l})}\\
\E[VU]
	&=AB\frac{\e^{-4\l}}{(1-\e^{-2\l})^2}\\
\E[VU^2]
	&=B\E[V/c_0]\E[(c_0U)^2]=B^2\frac{\e^{-2\l}}{1-\e^{-2\l}}\Big(C\frac{\e^{-4\l}}{1-\e^{-4\l}}+A^2\frac{2\e^{-6\l}}{(1-\e^{-2\l})(1-\e^{-4\l})}\Big)
\end{align*}
From the above computations and \ale{\eqref{coldiretti1} } we  get
\begin{align*}
\sigma_1^2(\l)
	&=\frac{4}{(1+2AB\frac{\e^{-2\l}}{1-\e^{-2\l}})^3}\times\\
	&\times\Big[ \frac{CD\e^{-4\l}}{1-\e^{-4\l}}+\frac{2\e^{-6\l}(A^2D+B^2C)}{(1-\e^{-2\l})(1-\e^{-4\l})}+\frac{4A^2B^2\e^{-8\l}}{(1-\e^{-2\l})^2(1-\e^{-4\l})}+\frac{AB\e^{-2\l}}{1-\e^{-2\l}}+\frac{2AB\e^{-4\l}}{(1-\e^{-2\l})^2}\Big]\,.
\end{align*}
Equivalently, we have
\begin{equation}\label{rey1}
\begin{split}
\sigma_1^2(\l)
	=\frac{4(\e^{2\l}-1)^3}{(\e^{2\l}-1+2AB)^3}\Big[ & \frac{CD}{\e^{4\l}-1}+\frac{2(A^2D+B^2C)}{(\e^{2\l}-1)(\e^{4\l}-1)}\\
	&+\frac{4A^2B^2}{(\e^{2\l}-1)^2(\e^{4\l}-1)}+\frac{AB}{\e^{2\l}-1}+\frac{2AB}{(\e^{2\l}-1)^2}\Big]\,.
\end{split}
\end{equation}
For $\sigma_2^2(\l)$ we need to calculate $\E[U\theta^n U]$ for $n\geq 1$. We first take  $n\geq 2$.  In this case we can write
\begin{equation}\label{sugar}
\begin{split}
\E[U\theta^n U]
	&=\E\Big[U\sum_{j=0}^{n-2} \frac{c_{n-j-1}}{c_n}\e^{-2\l(j+1)}\Big]
		+\E\Big[U\frac{c_{0}}{c_n}\e^{-2\l n}\Big]
		+\E\Big[U\sum_{j=n}^{\infty} \frac{c_{n-j-1}}{c_n}\e^{-2\l(j+1)}\Big]\\
	&=:C_1+C_2+C_3\,.
\end{split}
\end{equation}
The three terms are calculated as
\begin{align}
C_1
	&=\E[U]AB\sum_{j=0}^{n-2} \e^{-2\l(j+1)}
	=A^2B^2\frac{\e^{-2\l}}{1-\e^{-2\l}}\Big(\frac{\e^{-2\l}-\e^{-2\l n}}{1-\e^{-2\l}}\Big)\nonumber\\
	&
	=\E[U]^2-A^2B^2\frac{\e^{-2\l}}{(1-\e^{-2\l})^2}\e^{-2\l n} \label{cinico} \\
C_2
	&=AB\frac{\e^{-2\l}}{1-\e^{-2\l}}\e^{-2\l n}\nonumber \\
C_3
	&= B^2\e^{-2\l n}\E\Big[\sum_{i=0}^\infty c_{-i-1} \e^{-2\l(i+1)}\sum_{j=0}^{\infty} {c_{-j-1}}\e^{-2\l(j+1)}\Big]=B^2\e^{-2\l n}\frac{\E[U^2]}{D}\,.\nonumber
\end{align}
In the case $n=1$ \eqref{sugar} is again valid with the convention that $C_1:=0$. 
On the other hand, for $n=1$, the expression in  \eqref{cinico} is zero, hence the above formulas for $C_1, C_2, C_3$ are valid also in the case $n=1$.
Hence we can calculate
\begin{align*}
& \sum_{n\geq 1}\left(\E[U\theta^n U]-\E[U]^2\right)
	=\Big(-A^2B^2\frac{\e^{-2\l}}{(1-\e^{-2\l})^2}+
	AB\frac{\e^{-2\l}}{1-\e^{-2\l}}+B^2\frac{\E[U^2]}{D}
	\Big)
	\sum_{n=1}^\infty\e^{-2\l n}\\
	&=-\frac{A^2B^2\e^{-4\l}}{(1-\e^{-2\l})^3}+
	\frac{AB\e^{-4\l}}{(1-\e^{-2\l})^2}+ \frac{B^2C\e^{-6\l}}{(1-\e^{-4\l})(1-\e^{-2\l})}+\frac{2A^2B^2\e^{-8\l}}{(1-\e^{-2\l})^2(1-\e^{-4\l})}
\end{align*}
On the other hand, by the above computations of $E[U]$, $E[U^2]$, we have
\[
E[U^2]-E[U]^2=
CD \frac{\e^{-4\l}}{1-\e^{-4\l}}+\frac{2 A^2D \e^{-6\l}}{(1-\e^{-2\l})(1-\e^{-4\l})}-A^2B^2\frac{\e^{-4\l}}{(1-\e^{-2\l})^2}\,.
\]
Due to the above identities and \eqref{coldiretti2} we have 
\begin{equation}\label{rey2}
\begin{split}
v(\l)\sigma_{2}^2(\l)
	=&\frac{4(\e^{2\l}-1)^3}{(\e^{2\l}-1+2AB)^3}
	\Big[ \frac{CD}{\e^{4\l}-1}+\frac{2A^2D}{(\e^{2\l}-1)(\e^{4\l}-1)}-\frac{A^2B^2}{(\e^{2\l}-1)^2}\\
	&-\frac{2A^2B^2\e^{2\l}}{(e^{2\l}-1)^3}+
	\frac{2 AB }{(\e^{2\l}-1)^2}+ \frac{2 B^2C }{(\e^{4\l}-1)(\e^{2\l}-1)}+\frac{4A^2B^2 }{(\e^{2\l}-1)^2(\e^{4\l}-1)}\Big]\,.
\end{split}
\end{equation}
\ale{By   \eqref{coeff_diff}, summing the expressions \eqref{rey1} and \eqref{rey2},  we get $\s^2(\l)$ and in particular   \eqref{rey}.}

\bigskip

\noindent
{\bf Acknowledgements}. We thank P. Mathieu for useful discussions.

\end{document}